\documentclass[11pt]{amsart}
\usepackage{amssymb}

\theoremstyle{plain}
\newtheorem{thm}{Theorem}[section]
\newtheorem{prop}[thm]{Proposition}

\newtheorem{cor}[thm]{Corollary}
\theoremstyle{remark}
\newtheorem{remark}[thm]{Remark}

\theoremstyle{definition}

\newtheorem{defin}[thm]{Definition}
\newtheorem{ex}[thm]{Example}
\newtheorem{prob}[thm]{Problem}

\newcommand{\J}{\mathbb{J}}
\newcommand{\N}{\mathbb{N}}

\newcommand{\Z}{\mathbb{Z}}
\newcommand{\R}{\mathbb{R}}

\newcommand{\cF}{\mathcal F}
\newcommand{\cC}{\mathcal{C}}

\newcommand{\cB}{\mathcal{B}}

\newcommand{\cM}{\mathcal{M}}

\newcommand{\xt}{\tilde x}

\newcommand{\zt}{\tilde z}
\newcommand{\ft}{\tilde f}
\newcommand{\Zt}{\tilde Z}
\newcommand{\St}{\tilde S}
\newcommand{\Tt}{\tilde T}

\newcommand{\Pb}{\overline{P}}

\newcommand{\vp}{\varepsilon}

\newcommand{\supp}{\text{\rm supp}}
\newcommand{\coo}{c_{00}}
\newcommand{\Vol}{\text{\rm Vol}}

\newcommand{\Id}{\text{\rm Id}}

\newcommand{\ra}{\rangle}
\newcommand{\la}{\langle}

\def\hangbox to #1 #2{\vskip1pt\hangindent #1\noindent \hbox to
#1{#2}$\!\!$}

\title{Coefficient Quantization for Frames in Banach Spaces}

\author{P.~G.~Casazza}\address{Department of Mathematics\\ University of Missouri
\\ Columbia, Mo 65211 USA}\email{pete@math.missouri.edu}

\author{S.~J.~Dilworth}\address{Department of Mathematics\\ University of South Carolina\\
Columbia, SC 29208  USA} \email{dilworth@math.sc.edu}

\author{E. Odell}
\address{Department of Mathematics \\
The University of Texas\\1 University Station C1200\\
Austin, TX 78712  USA}
\email{odell@math.utexas.edu}
\author{Th. Schlumprecht}
\address{Department of Mathematics, Texas A\&M University\\
College Station, TX 77843, USA}
\email{thomas.schlumprecht@math.tamu.edu}

\author{A. Zs\'ak}
\address{School of Mathematics, University of Leeds, \\ Leeds, LS2 9JT, United Kingdom}
\email{zsak@maths.leeds.ac.uk}

\thanks{The research of
the  first, second, third and fourth author was supported by the NSF. The first, second, third, and fifth authors were supported by
the Linear Analysis Workshop at Texas A\&M University in  2007. All authors were supported by the Banff
 International Research Station}


\begin{document}

\baselineskip 16pt
\begin{abstract} Let $(e_i)$ be a fundamental system of a Banach space.
 We consider the problem of approximating  linear combinations of elements 
  of this system by linear combinations using quantized coefficients. We will
   concentrate on systems which are possibly redundant. Our model for this 
   situation will be frames in Banach spaces.
\end{abstract}
\maketitle

\markboth{P.~G.~CASAZZA et al}{COEFFICIENT QUANTIZATION FOR FRAMES
IN BANACH SPACES}

\tableofcontents
\section{Introduction}\label{S:1}

Hilbert space frames provide a crucial theoretical underpinning for
  compression,  storage and   transmission of signals because they provide
 robust and stable representation of vectors. They  also have applications in mathematics
 and engineering in a wide variety of areas including
 sampling theory \cite{AG}, operator theory \cite{HL},  harmonic analysis , nonlinear sparse approximation 
\cite{DE}, pseudo-differential operators \cite{GH}, and quantum computing \cite{EF}.

  In many
 situations  it is  useful to think of a signal 
as being a vector
$x$ in a Hilbert space 
and being  represented as a   (finite or infinite) sequence  
 $(<x_i,x>)_{i=1}^\infty$,
where  $(x_i)$ is a  {\em frame}, i.e.
a sequence in $H$ which satisfies for some $0<a\le b$,
\begin{equation}\label{E:1.1}
a\|x\|^2\le\sum |<x_i,x>|^2\le b\|x\|^2, \text{ whenever } x\in H.
\end{equation}
Since the sequence $(x_i)$ does not have to be (and usually is not) a  basis for $H$,
the representation of an $x\in H$ as the sequence   $(<x_i,x>)_{i=1}^\infty$
 includes some redundancy, which, for example, can be used  to correct
errors in transmissions \cite{GKK}.  Using a Hilbert space as the underlying space
has, {\it inter alia\/}, the advantage of an easy reconstruction formula.
Nevertheless, there are circumstances which make it necessary to leave the 
confines of a Hilbert space,
and generalize  frames to the category of Banach spaces. One such instance
occurs when we wish to replace the 
frame coefficients by {\em quantized  coefficients}, i.e.
by integer multiples of a  given $\delta>0$.

An example of such a situation is described by Daubechies and DeVore in  
\cite{DD}:
Let $f\in L_2(-\infty,\infty)$ be a {\em band-limited function}, to wit, 
the  support of the Fourier Transform $\hat{f}$ is 
contained in $[-\Omega,\Omega]$ for some $\Omega>0$.
For simplicity we assume  that $\Omega=\pi$.
Now we can think of $\hat{f}$ as an element of $L_2[-\pi,\pi]$, write
$\hat{f}$ on $[-\pi,\pi]$ as a series  in $e^{-i nx}$,  $n\in\Z$, and apply the 
inversion formula
for the Fourier  transform. This leads to the {\em sampling formula}
$$
f(x)=\sum_{n\in\Z} f(n)\frac{\sin(x\pi-n\pi)}{x\pi-n\pi}, \quad x\in\R.
$$
This series converges `badly'. In particular it 
is not absolutely convergent in general. Therefore
we  consider  some $\lambda>1$ and  think 
of the space  $L_2[-\pi,\pi]$ as being
embedded (in the natural way) into  $L_2[-\lambda\pi,\lambda\pi]$. 
The family of functions
$ (e^{-i nx/ \lambda})_{n\in\Z}$ 
forms an orthogonal basis for $L_2[-\lambda\pi,\lambda\pi]$, and
it can be viewed
as a frame for the `smaller' space $L_2[-\pi,\pi]$ (see section \ref{S:2}).  We write
$\hat{f}(\xi)= \sqrt{2\pi}\hat{\rho}(\xi) \cdot \hat{f}(\xi)$,  
 where $\hat{\rho}:\R\to[0,1/\sqrt{2\pi} ]$ is $\cC_\infty$,
$\hat{\rho}|_{[-\pi,\pi]}\equiv 1/\sqrt{2\pi}  $,  and
$\hat{\rho}|_{(-\infty,\-\lambda\pi]\cup[\lambda\pi,\infty)}\equiv 0$. 
Now we can express  $\hat{f}$ on
$[-\lambda\pi,\lambda\pi]$ as a  series in
 $ (e^{-i nx/ \lambda})$  and apply the inverse transform once again. 
This leads to the expansion
$$
f(x)=\frac1\lambda \sum_{n\in\Z} f\Big(\frac{n}{\lambda}\Big)
 \rho\Big({x}-\frac{n}{\lambda}\Big), \quad x\in\R,
$$
which not only converges faster, but is also absolutely unconditionally convergent, since
 $\hat \rho$ is $C_\infty$, and, thus, $\rho$ and all its derivatives are in $L_1(\R)$.

Now assume that $\|f\|_{L_\infty}\le 1$ (note that bandlimited functions are bounded 
in $L_\infty$).
It was shown in \cite{DD} that  the $\Sigma-\Delta$-{\em quantization algorithm} 
 can be used to
find  a sequence $(q_n)_{n\in\Z}\subset\{-1,1\}$ for which
$$ \Big\vert f(x)-\frac1\lambda\sum_{n\in\Z} 
  q_n \rho\Big(x-\frac{n}{\lambda}\Big)\Big\vert\le
\frac1\lambda \|\rho'\|_{L_1}, \text{ for }x\in\R. 
$$
This means that 
 our approximation does not hold in $L_2$ (and it need not for 
the applications at hand)
but it does hold in the Banach space $L_\infty$ (in fact in $C(\R)$).

We 
consider therefore a signal to be an arbitrary vector $x$ in a Banach space $X$ and
 ask if there is a dictionary $(e_i)$, e.g. some sequence $(e_i)$ whose
 span is dense in $X$,  so that $x$ can be approximated in norm, 
up to some $\vp>0$, by a linear combination 
  of the $e_i$'s using  only coefficients from a discrete alphabet, i.e. the integer multiples of some given $\delta$. 
 The case that $(e_i)$ is a non-redundant system, for example a basis, or, more generally,
 a total fundamental minimal  system, was treated in \cite{DOSZ}.  It was shown there, for example,
that if $(e_i)$ is a semi-normalized fundamental and
total  minimal system  which has the property  that for some $\vp,\delta>0$ every vector of the form
$x=\sum_{i\in E} a_i e_i$, with $E\subset \N$ finite, can 
 be $\vp$-approximated by a vector $\tilde x= \sum_{i\in E} \delta k_i e_i$,
 with $(k_i)\subset \Z$, then  $(e_i)$ must have a subsequence which is either
equivalent to  unit-vector basis of $c_0$, or to the summing basis for $c_0$. Conversely,
every separable Banach space $X$ containing $c_0$ admits such 
  a total fundamental minimal  system.
 
In this work we will concentrate on redundant dictionaries. 
Our  model for redundant  dictionaries will be  frames in Banach spaces. In section \ref{S:2} we shall
 recall their definition  and make some elementary observations.  Before we tackle the problem of
 coefficient quantization with respect to frames, we first have to ask ourselves what exactly we mean
 by a {\em meaningful coefficient quantization}.  In section \ref{S:3} we recall the notion
 {\em Net Quantization Property } (NQP) as introduced for fundamental systems in \cite{DOSZ}.
 We shall then  present several examples of systems which formally satisfy the NQP, but on the other hand  clearly do not
 accomplish the goals of quantization, namely data compression and easy reconstruction. These 
examples will lead us to a  notion of quantization which is more restrictive, and more meaningful,
 in the case of redundant systems. 
 
In section \ref{S:4} we ask under which circumstances one can approximate a vector in a Banach space
$X$ by a vector with quantized coefficients which are bounded in some {\em associated sequence space}
 $Z$ with a basis $(z_i)$  (see Definition \ref{D:4.1}).
 If $Z$ has non trivial lower estimates this is   only possible 
 if 
 one reconciles with the fact  that the length of the frame increases exponentially with the dimension of the underlying space. 
 We shall show  this type of quantization cannot happen if 
  $(z_i)$ satisfies   nontrivial  lower  and  upper estimates. The proof 
of these facts utilizes volume arguments
 and  must therefore be formulated first in the finite-dimensional case. An infinite-dimensional
 argument proves directly that the associated space $Z$ with a semi-normalized basis $(z_i)$ cannot be reflexive. In particular, there is no semi-normalized frame $(x_i)$  for an infinite-dimensional Hilbert space so that
 for some choice of $0<\vp,\delta<1$ and $C\ge 1$, every $x\in H$, $\|x\|=1$, can be $\vp$-approximated
by a vector $\tilde x=\sum \delta  k_i x_i$, with $(k_i)\subset \Z$ and $\sum \delta^2 k_i^2 \le C$.

In section \ref{S:5} we consider conditions under which 
  an $n$-dimensional space admits,  for given $\vp,\delta>0$ and $C\ge 1$,
 a finite frame $(x_i)_{i=1}^N$, so that every element in the {\em zonotope}
$\{\sum_{i=1}^N a_i x_i : |a_i|\le 1\}$ can be $\vp$-approximated by some
element from $\{\sum_{i=1}^N \delta k_i x_i : k_i\in \Z, |k_i|\le C/\delta \}$.
  Using results from convex geometry we shall show that this is only possible
 for spaces  $X$ with trivial cotype. Among others, we provide an answer
 to a question raised in \cite{DOSZ} and prove that $\ell_1$ does not have a semi-normalized basis with the NQP. 

In the  final section  we will state some open problems.

 All Banach spaces are considered to be spaces  over the real field $\R$.
$S_X$ and $B_X$, denote the unit sphere and the unit ball of a Banach space $X$,
 respectively.
For a set $S$ we denote by $\coo(S)$, or simply  $\coo$, if $S=\N$, the set
 of all families $x=(\xi_s)_{s\in S}$ with finite support,
 $\supp(x)=\{s\in S: \xi_s\not= 0\}$. The unit vector basis of $\coo$, as well as the unit vector basis
 of $\ell_p$, $1\le p<\infty$, and $c_0$ is denoted by $(e_i)$.

A {\em Schauder basis}, or simply a {\em basis}, of a Banach space $X$ is a sequence $(x_n)$, which has the property
that every $x$ can be uniquely written as a norm converging series $x=\sum a_i x_i$.
 It follows then from the Uniform Boundedness Principle that
 the coordinate  functionals $(x_n^*)$, 
$$x^*_n: X\to \R,\quad \sum a_i x_i\mapsto a_n$$ are bounded (cf.\cite{FHHMPZ}) and
the projections $P_n$, with 
$$P_n: X\to X,\quad x=\sum a_i x_i\mapsto \sum_{i=1}^n a_i x_i,\text{ for $n\in\N$}$$
are continuous and uniformly bounded in the operator norm. We call
$C=\sup_{n\in\N} \|P_n\|$ the {\em basis constant of $(x_i)$} and
$K=\sup_{0\le m\le n} \|P_n-P_m\|$ $(P_0\equiv 0$) the  {\em projection constant of $(x_i)$}.
Note that $C\le K\le 2C$. We call  $(x_n)$ {\em monotone} if $C=1$ and {\em bimonotone} if also $K=1$.
A basis $(x_n)$ is called {\em unconditional} if for any $x\in X$ the unique representation
 $x=\sum a_nx_n$ converges unconditionally. This is equivalent (cf. \cite{FHHMPZ}) to the property
that for all $(a_i)\in \coo$ 
$$K_u=\sup\Big\{\Big\| \sum \pm a_i x_i\Big\|:  \Big\|\sum  a_i x_i \Big\|=1 \Big\}<\infty.$$

If $X$ is a finite  dimensional space we can represent it isometrically as $(\R^n,\|\cdot\|)$ where 
$\|\cdot\|$ is a norm function on $\R^n$. With this representation we
 consider the Lebesgue measure of a measurable set $A\subset \R^n$ and  denote
  it by $\Vol(A)$. Of course $\Vol(A)$ depends on the  representation of $X$.
Nevertheless, if we only consider certain ratios of volumes this is not the case.
 Therefore, the quotient $\Vol(A)/\Vol(B)$ is well defined even in abstract finite dimensional
 spaces without any specific representation.       

\section{Frames in Hilbert spaces and Banach spaces}\label{S:2}

In this section we give a short review of the concept of frames in Banach spaces,
 and  make some preparatory observations. Let us start with the  well known notion   of  Hilbert space
 frames.

\begin{defin}\label{D:2.1} Let $H$ be a  (finite or infinite dimensional) Hilbert space.
A sequence $(x_j)_{j\in \J}$ in  $H$, $\J=\N$ or $\J=\{1,2,\ldots, N\}$, for some $N\in\N$, is called a {\em frame of $H$}
or {\em Hilbert frame for $H$}
if there are $0<a\le b<\infty$ so that
\begin{align}\label{E:2.1.1}
 a\|x\|^2\le \sum_{j\in \N} |\la x,x_j\ra|^2\le b \|x\|^2 \text{ for all $x\in H$}.
\end{align}
\end{defin}
For a frame   $(x_j)_{j\in\J}$  of $H$ we consider the operator
\begin{equation*}
\Theta: H\to \ell_2(\J), \qquad x\mapsto  (\la x,x_j\ra)_{j\in \J}, 
\end{equation*}
 its adjoint
 \begin{equation*}
\Theta^*:\ell_2(\J)\to H, \qquad (\xi_j)_{j\in \J}\mapsto \sum_{j\in \J} \xi_j x_j 
\end{equation*}
and their product
\begin{equation*}
	I=\Theta^* \circ \Theta: H\to H,\qquad x\mapsto \sum_{j\in\J} \la x,x_j\ra x_j.
\end{equation*}
Since   
\begin{equation*}
a \|x\|^2 \le   \sum_{j\in \N} |\la x,x_j\ra|^2= \Big\la x,   \sum_{j\in \N} \la x,x_j\ra x_j\Big\ra
    = \la x,I(x)\ra\le b\|x\|^2,
\end{equation*}
$I$ is a positive and invertible operator with  $a\text{Id}_H\le I\le  b \text{Id}_H$ and thus, 
\begin{align*}
&x=I^{-1}\circ I (x) = \sum_{j\in \N} \la x,x_j\ra I^{-1}(x_j),\text{ or }\\
\label{E:1.4}
&x=I\circ I^{-1} (x) =  \sum_{j\in \N} \la I^{-1}(x),x_j\ra x_j = \sum_{j\in \N} \la x,I^{-1}(x_j)\ra x_j.
\end{align*}

For an introduction to the theory of Hilbert space frames we refer the reader to \cite{Ca1} and \cite{Ch}.
  We follow \cite{HL} and \cite{CHL} for the generalization of   frames to Banach spaces.

\begin{defin}\label{D:2.2}(Schauder Frame)
Let $X$ be a  (finite or infinite dimensional) separable Banach space. A  sequence
 $(x_j,f_j)_{j\in\J}$, with  $(x_j)_{j\in\J}\subset X$,
 $(f_j)_{j\in\J}\subset X^*$, and $\J=\N$ or $\J=\{1,2\ldots N\}$, for some $N\in\N$, is called a
 {\em (Schauder) frame  of $X$} if
for every $x\in X$
\begin{equation}\label{E:2.2.1}
x=\sum_{j\in\J} f_j(x) x_j \,.
\end{equation}

In case that $\J=\N$, we mean that the series in \eqref{E:2.2.1} converges
 in norm, i.e. that $x=\lim_{n\to\infty}\sum_{j=1}^n f_j(x) x_j$. 

 An {\em unconditional frame of $X$}
  is a frame
 $(x_i,f_i)_{i\in\N}$ for $X$
 for which  the convergence in (\ref{E:2.2.1}) is unconditional.

 We call a frame $(x_i,f_i)$ {\em bounded } if
$$\sup_i\|x_i\|<\infty \text{ and }\sup_i\|f_i\|<\infty,$$
 and semi-normalized if $(x_i)$ and $(f_i)$ are semi-normalized, i.e. if 
$0<\inf_i\|x_i\|\le \sup_i\|x_i\|<\infty$ and $0<\inf_i\|f_i\|\le \sup_i\|f_i\|<\infty$.
\end{defin}

In the following Remark we make some easy observations.
\begin{remark}\label{R:2.3}
Let $(x_i,f_i)_{i\in\N}$  be a frame of $X$.
\begin{enumerate}
\item[a)]  If  $\inf_{i\in\N} \|x_i\|>0$, then
$\displaystyle f_i{\mathop{\to}^{w^*}} 0$ as $i\to\infty$.
\item[b)] Using the Uniform Boundedness Principle we deduce that
$$K=\sup_{x\in B_X}\sup_{m\le n} \Big\|\sum_{i=m}^n f_i(x) x_i\Big\|<\infty.$$
This implies that  if $\inf_{i\in\N} \|x_i\|>0$ then $(f_i)$ is bounded and if
  $\inf_{i\in\N} \|f_i\|>0$   then $(x_i)$ is bounded.

We call $K $ the {\em projection constant of  $(x_i,f_i)$}. The projection constant
 for finite frames is defined accordingly.

\item[c)] For all $f\in X^*$ and $x\in X$ it follows that
 $$f(x)=f\Big(\sum_{i=1}^\infty f_i(x) x_i\Big)=\sum_{i=1}^\infty f_i(x) f(x_i)
=\lim_{n\to\infty} \Big(\sum_{i=1}^nf(x_i) f_i\Big)(x),$$ and, thus,
$$f=w^*-\sum_{i=1}^\infty f(x_i) f_i.$$
Moreover, for $m\le n$ in $\N$ it follows that
\begin{align}\label{E:2.3.1}
\Big\| \sum_{i=m}^n f(x_i) f_i\Big\|&= \sup_{x\in
B_X}\Big|\sum_{i=m}^n f(x_i) f_i(x)\Big|
\le
 \|f\| \sup_{x\in B_X}\Big\|\sum_{i=m}^n f_i(x) x_i\Big\|\le K \|f\|,
 \end{align}
and
\begin{align}\label{E:2.3.2}
\Big\| \sum_{i=m}^\infty f(x_i) f_i\Big\|&=\sup_{x\in B_X}\Big|\sum_{i=m}^\infty f(x_i) f_i(x)\Big|\\
&=\sup_{x\in B_X}f\Big(\sum_{i=m}^\infty  f_i(x) x_i\Big)\notag\\
&\!\!\!\!\begin{cases} \le \sup_{ z\in\text{\rm span}(x_i:i\ge m), \|z\|\!\le\!
K }f(z)\!=\!
 K\|f|_{\text{\rm span}(x_i:i\ge m)}\|\\
\ge\|f|_{\{f_1,f_2,\ldots f_m\}^\perp}\|.
\end{cases}\notag
\end{align}
\item[d)] If $(x_i,f_i)$ is an unconditional frame it follows from the Uniform Boundedness Principle that
$$K_u=\sup_{x\in B_X}\sup_{(\sigma_i)\subset\{\pm 1\}} \Big\|\sum \sigma_i f_i(x) x_i\Big\|<\infty.$$
We call $K_u$ the  {\em unconditional constant of  $(x_i,f_i)$}.
\end{enumerate}
\end{remark}
The
following  Proposition  is    a slight variation of
 \cite[Theorem 2.6]{CHL}.

\begin{prop}\label{P:2.4}
 Let $X$ be a separable  Banach space and let $(x_i)_{i\in\J}\subset X$
 and $(f_i)_{i\in\J}\subset X^*$, with  $\J=\N$ or $\J=\{1,2,\ldots N\}$ for some $N\in\N$.
\begin{enumerate}
\item[a)] $(x_i,f_i)_{i\in\J}$ is
 a Schauder frame of $X$ if and only if
there is
 a Banach space $Z$ with a Schauder  basis $(z_i)_{i\in\J}$ and corresponding
 coordinate functionals $(z_i^*)$,  an isomorphic embedding $T:X\to Z$
  and a bounded linear surjective map
 $S:Z\to X$, so that  $S\circ T=Id_X$ (i.e. $X$ is isomorphic to a complemented
 subspace of $Z$), and  $S(z_i)=x_i$,  for $i\in\J$, and $T^*(z^*_i)=f_i$,  for $i\in\J$, with $x_i\not=0$.

Moreover $S$ and $T$ can be chosen so that $\|S\|=1$ and $\|T\|\le K$, where $K$ is the projection constant
 of $(x_i,f_i)$, and $(z_i)$ can be chosen to be a bimonotone basis with $\|z_i\|=\|x_i\|$ if $i\in\J$, with $x_i\not=0$.
\item[b)]   $(x_i,f_i)_{i\in\J}$ is
 an unconditional  frame  of $X$
 if and only if there  is 
 a Banach space $Z$ with an unconditional   basis $(z_i)$ and
corresponding
 coordinate functionals $(z_i^*)$,  an isomorphic embedding $T:X\to Z$
  and a surjection
 $S:Z\to X$, so that $S\circ T=Id_X$, $S(z_i)=x_i$,  for $i\in\J$, and $T^*(z^*_i) =f_i$ for $i\in\J$, with $x_i\not=0$.
\end{enumerate}
\end{prop}
\begin{proof}  (a) part ``$\Rightarrow$'' \,\,
 Assume that $(x_i,f_i)_{i\in\J}$ is a frame of $X$ and let $K$ be the projection constant of  $(x_i,f_i)_{i\in\J}$. We
put $\tilde \J=\{i\in\J: x_i\not=0\}$, denote the unit vector basis of $\coo(\J)$ by $(z_i)$
and define
 on $c_{00}(\J)$ the following norm $\|\cdot\|_Z$. 
\begin{equation}\label{E:2.4.1}
\Big\|\sum_{i\in\J} a_iz_i\Big\|_Z=\max_{m\le n} \Big\|\sum_{i\in\tilde\J\cap\{m,m+1,\ldots, n\}} a_ix_i\Big\|_X+
\Big(\sum_{i\in \J\setminus \tilde\J} a_i^2\Big)^{1/2}
\text{ for }(a_i)\subset \R.
\end{equation}
It follows easily that $(z_i)$ is a bimonotone basic sequence and, thus, a basis of  the completion of $\coo(\J)$ with respect
to $\|\cdot\|_Z$, which we denote by $Z$.

The map 
$$S:Z\to X,\quad \sum a_j z_j\mapsto  \sum a_j x_j,$$
is linear and bounded  with $\|S\|=1$. Secondly, define 
$$T:X\to Z,\quad x=\sum_{i\in\J} f_i(x) x_i =\sum_{i\in\tilde\J} f_i(x) x_i \mapsto \sum_{i\in\tilde\J} f_i(x) z_i.$$
Remark \ref{R:2.3} (b) yields for $x\in X$ 
\begin{align*}
 \Big\|\sum_{i\in\tilde\J} f_i(x) z_i\Big\|_Z&=\sup_{m\le n}
\Big\| \sum_{i\in\tilde\J\cap\{m,m+1,\ldots, n\}} f_i(x) x_i\Big\|_X\\
&=
\sup_{m\le n}
\Big\| \sum_{i\in\J\cap\{m,m+1,\ldots, n\}} f_i(x) x_i\Big\|_X\le K\|x\|, 
\end{align*}
and, thus,  that $T$ is linear and bounded with $\|T\|\le K$. Clearly it follows that $S\circ T=Id_X$,
 which implies that $T$ is an isomorphic embedding and that $S$ is a surjection.
Finally, if $(z_i^*)$ are the coordinate functionals of $(z_i)$ we deduce for $x\in X$ and $i\in\J$ that
$$T^*(z_i^*)(x)= z^*_i(T(x))=\begin{cases} f_i(x) &\text{ if $x_i\not=0$}\\
                                           0      &\text{ if $x_i=0$,}  \end{cases}    $$
which finishes the proof of  `` $\Rightarrow$''. 

In order to show the converse  in (a), assume that
$Z$ is a space with a basis $(z_i)_{i\in\J}$ and that $S:Z\to X$ is a bounded linear surjection, and $T: X\to Z$ 
 an isomorphic embedding, with $S\circ T=Id_X$. Put $x_i=S(z_i)$ and $f_i=T^*(z^*_i)$, for $i\in \J$. Then for $x\in X$,
$$x= S\circ T(x)=S\Big(\sum z^*_i(T(x)) z_i\Big)=\sum T^*(z^*_i)(x) S(z_i)= \sum f_i(x) x_i,$$
which implies that $(x_i,f_i)_{i\in\J}$ is a frame of $X$.

For the  proof of (b) we replace \eqref{E:2.4.1} by
\begin{equation}\label{E:2.4.2}
\Big\|\sum_{i\in\J} a_iz_i\Big\|_Z=
\max_{(\sigma_i)\subset \{\pm1\}} \Big\|\sum_{i\in\tilde\J} \sigma_i a_ix_i\Big\|_X+
\Big(\sum_{i\in \J\setminus \tilde\J} a_i^2\Big)^{1/2}.
\end{equation}
and note that  arguments  similar to those in the proof of (a)  yield (b).
\end{proof}

\begin{defin}\label{D:2.5} Let $(x_i,f_i)$ be    a frame of a Banach space $X$ 
and let  $Z$ be a space with a basis
 $(z_i)$ and   corresponding coordinate functionals $(z_i^*)$.
 We call $\big(Z,(z_i)\big)$ an
 {\em associated space to $(x_i,f_i)$} or
 a {\em sequence space associated  to $(x_i,f_i)$}    and $(z_i)$ an  {\em  associated basis},
 if
 \begin{align*}
 &S:Z\to X,\quad\sum a_i z_i\mapsto \sum  a_i x_i \text{ and }
 &T:X\to Z,\quad x=\sum f_i(x) x_i\mapsto \sum_{x_i\not=0}  f_i(x) z_i
\end{align*}
 are  bounded operators. We call $S$ the {\em associated
   reconstruction operator } and $T$ the  {\em  associated
   decomposition operator} or {\em analysis operator}.
   
  In this case, following \cite{Gr} we call  the triple $ \big((x_i),(f_i), Z\big)$ an {\em atomic decomposition} of $X$. 
   
\end{defin}
\begin{remark}\label{R:2.6a}
By Proposition \ref{P:2.4}  the property  of Banach space $X$ to admit a frame 
  is equivalent to the property of $X$ being  isomorphic to a complemented subspace of  a space
 $ Z$ with basis. It was shown independently  by  Pe{\l}czy{\'n}ski \cite{Pe}
 and Johnson, Rosenthal and Zippin \cite{JRZ} (see also \cite{Ca2}[Theorem 3.13]) that the later property is equivalent
 to  $X$ having the {\em Bounded Approximation Property}. we $X$ is said to have the {\em Bounded Approximation Property} if there is a $\lambda\ge 1$, so that for every $\vp>0$ and every compact set $K\subset X$ there is a finite rank operator  $T:X\to X$ with $\|T\|\le \lambda$ so that
$\|T(x)-x\|\le \vp$  whenever $x\in K$.  
\end{remark}

\begin{remark}\label{R:2.6} Let  $(x_j)_{j\in\J}$ be a 
Hilbert frame   of a Hilbert space $H$ and let $\Theta$ and $I$ be defined as in the paragraph following
   Definition \ref{D:2.1}.
 We choose
$Z$ to be $\ell_2(\J)$, 
$S=\Theta^*$ and
$$T=\Theta\circ I^{-1}: H\to\ell_2,\quad x\mapsto \sum_{j\in\J} \la I^{-1}(x), x_j \ra  e_j =  
  \sum_{j\in\J} \la x, I^{-1}( x_j )\ra  e_j$$
and observe that   $S\circ T=\Id_H$, and for $j\in\J$  it follows that
$S(e_j)=\Theta^*(e_j)=x_j$,  and
$$T^*(e_j)(x)= \la x, I^{-1} (x_j) \ra, \text{ ($x\in H$) and, thus }T^*(e_j)=  I^{-1}(x_j) .$$
Thus, if $(x_i)$ is a Hilbert frame, then $((x_i),(I^{-1}(x_i))$ is a Schauder frame for
 which  $Z=\ell_2(\J)$ together with its unit vector basis is an associated space. 
 
Conversely, let $(x_i,f_i)$ be a  Schauder frame of a Hilbert space $H$ and assume that
$Z=\ell_2(\J)$
 with its unit vector basis  is an associated space. Denote by  $T: H\to\ell_2(\J)$  and $S:\ell_2(\J)\to H$
the associated decomposition, respectively reconstruction operator. 
Then it follows that for all $x\in H$
$$\sum \la x_i,x\ra^2=\sum \la S(e_i),x\ra^2=\sum \la  e_i,S^*(x)\ra^2= \|S^*(x)\|^2.$$
Thus, since $S^*$ is an isomorphic embedding of $H$ into $\ell_2$, it follows that
 $(x_i)$ is a Hilbert frame.

\end{remark}

In the following observation  we show, that  we can always {\em expand a  frame by a 
bounded linear operator}.

\begin{prop}\label{P:2.7}
Let $(x_i,f_i)$ be a frame of a Banach space $X$ 
and let $Z$ be a space with a basis $(z_i)$ which is associated to  $(x_i,f_i)$.
Furthermore assume that $Y$ is another space with a basis  $(y_i)$
  and let  $V:Y\to X$ be linear and bounded.

Let $\Zt=Z\oplus_\infty Y$ and define $(\zt_i)\subset \Zt$, 
$(\xt_i)\subset X$ and $(\ft_i)\subset X^*$ by
$$
\zt_i=\!\begin{cases} (z_{i/2},\lambda_{i/2} y_{i/2})\\ (z_{(i+1)/2},-\lambda_{(i+1)/2}y_{(i+1)/2})\end{cases}
\xt_i=\!\begin{cases} x_{i/2}+\lambda_{i/2} V(y_{i/2}) &\text{ if $i$ even} \\ x_{(i+1)/2}-\lambda_{(i+1)/2}V(y_{(i+1)/2})&\text{ if $i$ odd,}\end{cases}$$
$$\ft_i=\!\begin{cases} \frac12 f_{i/2} &\text{ if $i$ even} \\ \frac12f_{(i+1)/2}&\text{ if $i$ odd,}\end{cases}
$$
where $\lambda_i=\|z_i\|/\|y_i\|$, for $i\in\N$.
Then $(\xt_i,\ft_i)$ is a frame of $X$, $(\zt_i)$ is a basis for $\Zt$ and
 $(\Zt,(\zt_i))$ is an associated space for  $(\xt_i,\ft_i)$.

\end{prop}

\begin{proof} Let $T:X\to Z$  and $S:Z\to X$  be the associated decomposition and reconstruction  operator, respectively.
Note that the operators 
\begin{align*}
\St: Z\oplus_\infty Y\to X,\quad (z,y)\mapsto S(z)+V(y)\text{ and }
\Tt:X\to Z\oplus_\infty Y,\quad x\mapsto (T(x),0)\end{align*}
are bounded and linear and that $\St\circ \Tt=Id_X$ and 
   $\St(\zt_i)=\xt_i$, for $i\in\N$. It is easy to verify that $(\zt_i)$ is a basis of $\Zt$
 for which its coordinate functionals $(\zt_i^*)$ are given by
 (denote the coordinate functionals of $(y_i)$ by $(y^*_i)$)
 $$\zt_i^*=\begin{cases} \frac12(z_{i/2}^*,\frac1{\lambda_i}y_{i/2}^*)&\text{ if $i$ even}\\  \frac12(z_{(i+1)/2}^*,-\frac1{\lambda_i}y_{(i+1)/2}^*)
&\text{ if $i$ odd.}\end{cases}$$
it follows for $x\in X$ that
$$\tilde T^*(\tilde z^*_i)(x)=\zt^*_i(\Tt(x))= \left\{ 
 \begin{matrix} f_{i/2}(x)/2 &\text{\ if $i$ even} \\ 
  f_{(i+1)/2}(x)/2 &\text{if $i$ odd}\end{matrix}\right\}=\ft_i(x),$$

which yields $\Tt^*(\zt_i^*)=\ft_i$,   for $i\in\N$.
Thus, the claim follows from Proposition \ref{P:2.4}.

\end{proof}

\section{Three Examples}\label{S:3}

In \cite{DOSZ} the following notion of quantization was introduced and studied for non redundant 
 systems.

\begin{defin}\label{D:3.1} Let $(x_i)_{i\in\J}$ be a 
 fundamental system for $X$, with $\J=\N$ or $\J=\{1,2,\ldots N\}$, for some
 $N\in\N$ (i.e. $\overline{\text{span}(e_i:i\in\J)}=X$) and let $\varepsilon>0$ and $\delta>0$ be given.
We say that $(x_i)_{i\in\J}$ has the $(\varepsilon,\delta)$-\textit{Net
Quantization Property} (abbr.\ $(\varepsilon,\delta)$-NQP) if for
 any
 $x = \sum_{i\in E} a_ix_i \in X$, $E\subset \J$ finite,
  there exists a sequence $(k_i) \subset \mathbb{Z}$,
  $\supp(k_i)=\{i\in\N: k_i\not=0\}$   is finite, such that
\begin{equation}\label{E:3.1}
 \Big\|x - \sum k_i \delta x_i\Big\| \le \varepsilon. \end{equation}
We say that $(x_i)$ has the NQP if $(x_i)$ has the $(\varepsilon,
\delta)$-NQP for some $\varepsilon>0$ and $\delta>0$.
\end{defin}
When we ask whether or not in a certain representation of vectors the coefficients
can be replaced by quantized coefficients, we are often interested in memorizing 
data as economically as possible, and reconstructing them with as little error as possible.
 With this in mind, we will exhibit in this section several examples, which show that
  it is not always meaningful to apply the notion of NQP word for word to
  redundant systems like frames.  These examples will  then also guide
  us to more appropriate  quantization concepts for frames.

The first example is a {\em tight} Hilbert frame $(f_i)$ in $S_{\ell_2}$ (i.e. $a=b$)
consisting of normalized vectors so that 
  for every $x\in\ell_2$ there is a sequence $(k_i)\subset \Z$
  so that
  $\|x-\sum_{i\in\N} k_i f_i\|<1$.
The second example is  a  semi-normalized Hilbert frame $(f_i)$ in $\ell_2$ which
has the property that for every  $x\in\ell_2$ there is a sequence
$(k_i)\subset \Z$
  so that
  $\|x-\sum_{i\in\N} k_i f_i\|<1$ and $(k_i)$ has the additional
  property that $\max_{i\in\N} |k_i|\le 1$, if $x\in B_{\ell_2}$.
 The third example is a Schauder frame  $(f_i)$ of $\ell_2$ which
 has the property that for every  $x\in\ell_2$ there is a sequence
$(k_i)\subset \Z$, so that not only   $\max_{i\in\N} |k_i|\le
\|x\|$
  and
  $\|x-\sum_{i\in\N} k_i f_i\|<1$, but so that also
   the support of $(k_i)$, i.e. the set
    $\{i\in\N: k_i\not=0\}$ is uniformly bounded.
\begin{ex}\label{Ex:3.2} Let $0<\vp_i<1/2$. For $i\in\N$
define  the following vectors $f_{2i-1}$ and  $f_{2i}$ in
$S_{\ell_2}$.
\begin{equation*}
f_{2i-1}=\sqrt{1-\vp_i^2} e_{2i-1}+\vp_i e_{2i}\text{ and }
f_{2i}=-\vp_i e_{2i-1} +\sqrt{1-\vp_i^2} e_{2i}.
\end{equation*}
Clearly $(f_i)$ is an orthonormal basis for $\ell_2$ and we let
$\cF=\{e_i:i\in\N\}\cup\{f_i:i\in\N\}$.
Then $\cF$ is a tight frame (as is  any  finite union of orthonormal
bases) and the sequence $(z_i)$ with
\begin{align*}
z_{2i-1}&=e_{2i-1}-f_{2i-1}=(1-\sqrt{1-\vp_i^2}) e_{2i-1}-\vp_i e_{2i}\\
z_{2i}&=e_{2i}-f_{2i} =\vp_i e_{2i-1}+(1-\sqrt{1-\vp_i^2})
e_{2i},\text{ for $i\in\N$ }
\end{align*}
is an orthogonal basis and $\|z_{2i-1}\|=  \|z_{2i}\|=O(\vp_i) $.
Thus, if $(\vp_i)$ converges fast enough
 to 0 it follows that for any $x\in\ell_2$ there is a family $(k_i)\subset \Z$, with $|\supp(k_i)|<\infty$,
so that 
$$\Big\|x-\sum k_i z_i\|= \Big\|x-\sum k_i (e_i-f_i)\Big\|<1.$$
\end{ex}

\begin{ex}\label{Ex:3.3} Our second example is  a semi-normalized  Hilbert frame 
 $(x_i)$  in $\ell_2$  so that  
$$D=\big\{\sum_{i=1}^\infty k_i x_i : k_i\in\{-1,0,1\} \text{ and }\{i:k_i\not=0\} \text{ is finite}\big\}$$
is dense in $B_{\ell_2}$.

 Put $ (c_i)_{i=1}^\infty=(1/2^i)$ and
  partition the unit vector basis $(e_i)$ of $\ell_2$
into infinitely many subsequences of infinite length, say
$(e(i,j)\!:\!i,j\!\in\!\N)$.
 Then
 our frame $(f_k)$  is defined to be the sequence:
 \begin{align*}
   &f_1=c_1e_1+e(1,1),f_2=e(1,1),\\
 &f_3=c_1e_2+e(1,2), f_4=e(1,2), f_5=c_2 e_1+e(2,1),f_6=e(2,1),\\
 &f_7\!=\!c_1 e_3+e(1,3),f_8\!=\!e(1,3),f_9\!=\!c_2 e_2+e(2,2),\\
 & \qquad\qquad \qquad\qquad  f_{10}\!=\!e(2,2),f_{11}\!=\!c_3 e_1+e(3,1),f_{13}\!=\!e(3,1),\\
 &f_{14}=c_1 e_4+e(1,4),f_{15}=e(1,4),\ldots,f_{20}=c_4 e_1+e(4,1),f_{21}=e(4,1),\\
 &\vdots
\end{align*}
Note that  the set of vectors  $x\in B_{\ell_2} $ of the form
$$x=\sum_{i,j\in\N} \vp(i,j) c_i e_j=\sum_{i,j\in\N} \vp(i,j)\big(c_i e_j +e(i,j)\big) -  \vp(i,j) e(i,j) $$
where $(\vp(i,j))\subset\{-1,0,1\}$ so that the set $\{i,j\in\N:
\vp(i,j)\not=0\}$ is finite, are dense in $B_{\ell_2}$.
 This implies that  every  $x\in B_{\ell_2}$ is the limit of vectors $(x_n)$  with
 $$(x_n)\subset \big\{\sum \vp_i f_i: (\vp_i)\subset\{-1,0,1\},  \{i\in\N: \vp_i\not=0\} \text{ is finite}\big\}.$$

The sequence $(f_n)$ is a frame. Indeed for any $x=\sum x_i e_i\in
\ell_2$ we have
\begin{align*}
\sum\langle f_i,x\rangle^2&=\sum_{i\text{ even}}\langle f_i,x\rangle^2+\sum_{i\text{ odd}}\langle f_i,x\rangle^2=
\sum x^2_i+\sum_{i,j\in\N} (c_i x_j+ \langle e(i,j),x\rangle)^2\\
&\begin{cases} \ge  \|x\|^2\\
\le \sum x^2_i+\sum_{i,j\in\N} 2c_i^2 x_j^2+ 2\langle
e(i,j),x\rangle^2\le (3+2\sum c_i^2)\|x\|^2.
\end{cases}
\end{align*}
\end{ex}

\begin{ex}\label{Ex:3.4}
We construct  a Schauder frame $(x_i,f_i)$  of
$\ell_2$ so that  $(x_n)$ is dense in $B_{\ell_2}$.

Let $(z_n)$ be dense in $B_{\ell_2}$ and choose for each $n\in\N$
$$x_{2n-1}=z_n+e_n, x_{2n}=z_n, f_{2n-1}=e_n \text{ and } f_{2n}=-e_n.$$
Clearly, for every $x\in\ell_2$ 
   $$x=\sum\langle e_i, x \rangle  e_i=\sum \langle f_{2n-1},x\rangle x_{2n-1}+
   \langle f_{ 2n},x\rangle x_{2n}$$
   (the above sum is conditionally converging).
   It follows that $\big((x_n),(f_n)\big)$ is a Schauder frame of
   $\ell_2$. It is clear that $(x_n)$ is  not a Hilbert
    frame.
\end{ex}

\begin{remark} All three examples satisfy the conditions (NQP)
 if we extend this notion word for word to frames. Nevertheless
  these examples do not satisfy our understanding of what
  quantization of coefficients should mean.

 In  Example \ref{Ex:3.2}  every $x\in\ell_2$ can be approximated by an
expansion with respect to a frame
 using only integer coefficients, but these coefficients might
 get arbitrarily large for elements in $B_{\ell_2}$. This means
   that we would need an infinite  alphabet to approximate
    vectors which are in $B_{\ell_2}$.
 Therefore it is not enough (as in the non redundant case) to assume that our frame is semi-normalized.
 
In the Examples \ref{Ex:3.3} and \ref{Ex:3.4} we achieve the
approximation of any vector in $\ell_2$ by  a quantized expansion whose
coefficients are bounded by a fixed multiple of the norm of the
vector, but in order to approximate even the vectors of a given
finite dimensional subspace
 (for example the space generated by  two elements of the unit vector basis
  elements of $\ell_2$) we need an infinite dictionary.
\end{remark}

\section{Quantization with $Z$-bounded coefficients}\label{S:4}

One way to
 avoid  examples like the ones mentioned  in Section \ref{S:3} 
 is to impose boundedness  conditions on the quantized coefficients 
 within an associated space $Z$.

\begin{defin}\label{D:4.1}
Assume  $(x_i,f_i)_{i\in\J}$, $\J=\N$ or $\J=\{1,2\ldots N\}$, for some $N\in\N$, is a frame of
 a Banach space $X$.  Let $Z$ be a space with  basis $(z_i)$ which is associated to   $(x_i,f_i)_{i\in\J}$.
Let $\vp,\delta>0$, $C\ge 1$.

We say that $(x_i,f_i)$ {\em satisfies  the  $(\vp,\delta,C)$-Net Quantization Property  with respect to} $(Z,(z_i))$  
 or $(\vp,\delta,C)$-$Z$-NQP, if for all $x\in X$ there exists  a sequence $(k_i)_{i\in\J}\subset \Z$ with finite support  
 so that
\begin{equation}\label{E:4.1.1}
\Big\|\sum k_i \delta z_i\Big\|_Z\le C\|x\| \text{ and }\Big\|x-\sum \delta k_i x_i\Big\|_X\le\vp.
\end{equation}

We say that $(x_i,f_i)$ {\em satisfies the NQP  with respect to $(Z,(z_i))$ } if  it satisfies  the
   $(\vp,\delta,C)$-$Z$-NQP for some choice of $\vp,\delta>0$ and $C\ge 1$.
\end{defin}

It is easy to see that the property $(\vp,\delta,C)$-NQP with respect to some associated space is 
homogenous in $(\vp,\delta)$, meaning that  a frame $(x_i,f_i)$  is  $(\vp,\delta,C)$-NQP if and only
if for some $\lambda>0$ (or for all $\lambda$) $(x_i,f_i)$  satisfies the   $(\lambda\vp,\lambda\delta,C)$-NQP.
The following result, analogous  to \cite[Theorem 2.4]{DOSZ},  shows that 
 it is enough to verify that one can quantize the coefficients of elments $x$ which are in $B_X$ to deduce the
 NQP.  
 
\begin{prop}\label{P:4.2} 
Assume that $(x_i)$ and $(z_i)$ are some sequences in Banach spaces $X$ and $Z$, respectively,
and assume that there are $C_0<\infty$, $\delta_0>0$  and  $0<q_0<1$, so that
 for all $x\in B_X$ there is a sequence $(k_i)\subset \Z$, $(k_i)\in\coo$  with
\begin{equation}\label{E:4.2.1}
\Big\| \sum \delta_0k_i z_i\Big\|\le C_0\text{ and } 
\Big\|x-\sum \delta_0 k_i x_i\Big\| \le q_0.
\end{equation}
Then  there are $\delta_1>0$, and $C_1<\infty$ only depending on 
 $\delta_0$, $q_0$ and $C_0$  so that 
 for all $x\in X$ there is a sequence $(k_i)\subset \Z$, $(k_i)\in\coo$,  with
\begin{equation}\label{E:4.2.1a}
\Big\| \sum \delta_1k_i z_i\Big\|\le C_1\|x\|\text{ and } 
\Big\|x-\sum \delta_1 k_i x_i\Big\| \le 1.
\end{equation}

\end{prop}
\begin{proof} Choose $n_1\in\N$ and $q_1$ so that
\begin{equation}\label{E:4.2.2}
\frac{n_1+1}{n_1} q_0=q_1<1.
\end{equation}
and put $\delta_1=\delta_0/n_1$.

We first claim that for any $0<\delta\le\delta_1$ and any $x\in B_X$
there is a sequence $(k_i)\in \Z^\N\cap \coo$ so that
\begin{equation}\label{E:4.2.3}
\Big\| \sum k_i \delta z_i\Big\|\le   2C_0\text{ and }
\Big\|x-\sum \delta k_i x_i\Big\| \le q_1.
\end{equation}
Indeed, let $\delta\le \delta_1$ and $x\in B_X$ and choose
$n\ge n_1$ in $\N$ so that 
$\frac{\delta_{0}}{n+1}<\delta\le\frac{\delta_0}n$ and $(k_i)\subset \Z$ so that
$$ \Big\| \sum k_i \delta_0 z_i\Big\|\le C_0\text{ and }
\Big\| \frac{\delta_0}{\delta(n+1)} x-\sum k_i \delta_0 x_i\Big\|\le q_0
$$
and, thus,  since $n\ge n_1$,
$$ \Big\| \sum k_i (n+1) \delta z_i\Big\|= 
 \Big\| \sum k_i \delta_0 z_i\Big\| \frac{\delta (n+1)}{\delta_0}\le \frac{n+1}{n} C_0\le2C_0$$
and 
$$\Big\|  x
-\sum k_i \delta (n+1) x_i\Big\|\le 
\Big\| \frac{\delta_0}{(n+1)\delta} x-\sum k_i \delta_0 x_i\Big\| \frac{\delta(n+1)}{\delta_0}\le
q_0\frac{\delta}{\delta_0}(n+1)\le q_1.$$

By induction on $n\in\N$ we show that for any $\delta\le q^{n-1}_1\delta_1$ and any $x\in B_X$ there is 
a
 $(k_i)\subset \Z$, $(k_i) \in c_{00}$, so that
\begin{equation}\label{E:4.2.4}
 \Big\| \sum k_i \delta z_i\Big\|\le 2C_0 \sum_{i=0}^{n-1} q^i_1\text{ and }
\Big\|  x-\sum k_i \delta x_i\Big\|\le q_1^n.
\end{equation}
For $n=1$ this is just \eqref{E:4.2.3}. Assume our claim to be true for $n$ 
and let $\delta \le \delta_1q^n_1$ and $x\in B_X$.  By our induction hypothesis, we can find
$(k_i)\subset \Z$, $(k_i)\in c_{00}$, so that
$$ \Big\| \sum k_i \delta z_i\Big\|\le 2C_0 \sum_{i=0}^{n-1} q^i_1\text{ and }
\Big\|  x-\sum k_i \delta x_i\Big\|\le q_1^n,$$
Since ${q_1^{-n}} \big( x-\sum k_i \delta x_i)\in B_X$ and since $\delta q_1^{-n}\le \delta_1$,
we can use our first claim and choose 
$(\tilde k_i)\in \Z^\N\cap \coo$ so that
\begin{equation*}
\Big\| \sum \tilde k_i \delta q_1^{-n} z_i\Big\|\le   2C_0\text{ and }
\Big\|{q_1^{-n}}\Big(x-\sum k_i \delta x_i\Big)-\sum \delta  q_1^{-n}\tilde k_i x_i\Big\| \le q_1,
\end{equation*}
and, thus
$$ \Big\| \sum (k_i+\tilde k_i) \delta z_i\Big\|\le 2C_0 \sum_{i=0}^{n-1} q^i_1+ 2C_0q_1^n,$$
and
$$\Big\|x-\sum \delta k_i  x_i-\sum \delta \tilde k_i x_i\Big\| \le q_1^{n+1},$$
which finishes the induction step.

Now define $C_1=2C_0\sum_{n=0}^\infty q^{n}_1$
 and let $x\in X$ be arbitrary.

If $\|x\|\ge 1$ (this is the only case left to consider) we choose $n\in\N$   with  
$q^{n}_1< \frac1{\|x\|}\le   q_1^{n-1}$
 and, by \eqref{E:4.2.4} we can choose $(k_i)\in \Z^\N\cap \coo$ so that
$$\Big\|\sum k_i \frac{\delta_1}{\|x\|} z_i\Big \|\le 2C_0 \sum_{i=0}^{n-1} q_1^i\le C_1
\text{ and }
\Big\|\frac{x}{\|x\|}-\sum \frac{k_i\delta_1}{\|x\|} x_i\Big\|\le q_1^n,$$
which yields 
$$\Big\|\sum k_i \delta_1 z_i\Big \|\le C_1q_1\|x\|\text{ and }
\Big\|x-\sum k_i \delta_1 x_i\Big \|\le  q_1^n\|x\|\le 1.$$
\end{proof}

In the following result we consider a finite frame $(x_i,f_i)_{i=1}^N$ of  a finite dimensional Banach space $X$, and exploit the fact that, if  $(x_i,f_i)_{i=1}^N$  has the  $(\vp,\delta_, C)$-NQP with respect
to some space $Z$ having a basis $(z_i)_{i=1}^N$, then  the value
$$\vp^{-n}=\Vol(B_X)/\Vol(\vp B_X) $$
must be smaller then the cardinality of the set
$$\cF_{(\delta,C)}(x_i)=\Big\{ \sum n_j \delta x_j: \Big\|\sum n_j \delta z_j\Big\|\le C\Big\}.$$
 
\begin{prop}\label{P:4.3} Assume $f,g:(0,\infty)\to(0,\infty)$ are strictly increasing functions so that
\begin{equation}\label{E:4.3.0}
\lim_{n\to\infty} f(n)=\infty \text{ and } \lim_{n\to\infty} \frac{g(n)\ln n}{n}=0,
\end{equation}
and  $C,B,R\ge 1$, and  $0\le \delta,\vp<1$
.

 Assume that $(x_i,f_i)_{i=1}^N$ is a frame of a Banach space  $X$ with 
$\dim(X)=n<\infty$ and that $Z$ is an $N$-dimensional space, $N\in\N$,  with basis $(z_i)_{i=1}^N$, which is associated to
$(x_i,f_i)_{i=1}^N$. Let $T:X\to Z$ and $S:Z\to X$ be the associated decomposition and reconstruction operator, 
denote  by $K_Z$  the projection constant of $(z_i)$.

Assume
\begin{align}
&\label{E:4.3.1}  \cF_{(\delta,C)}(x_i) \text{ is $\vp$-dense in } B_X, \text{ and }\\
&\label{E:4.3.2} f(\#A)\le \Big\|\sum_{j\in A} n_j z_j\Big\|,  \text{ whenever $A\subset\{1,2\ldots N\}$ and
 $(n_j)_{j\in A}\!\subset\! \Z\setminus\{0\}$}.
\end{align}
Then 
\begin{equation}
\label{E:4.3.3} \ln N\ge \frac{n\ln(1/\vp)}{f^{-1}(C/\delta)} -\ln \Big(\frac{4 K_Z C}{\delta} +1\Big).
\end{equation}

In addition there is an $n_0\in\N$ (depending on    $C$, $B$, $R$,  $\vp$, $\delta$, $f$ and $g$) so that if, moreover, $K_Z\le B$, $\|S\|\le R$ and
\begin{align}
&\label{E:4.3.4}  
 g(\#A)\ge \sup \Big\|\sum_{j\in A}  \pm z_j\Big\|,  \text{ whenever $A\subset\{1,2\ldots N\}$},
\end{align}
then $n\le n_0$.
\end{prop}
\begin{proof} 
  First note that, if $A\subset \{1,2,.., N\}$ and $(n_j)_{j\in A}\!\subset\!\Z\setminus\{0\}$
with $C\!\ge\!\big\|\sum_{j\in A} n_j \delta z_j\big\|_Z\!\ge \delta f(\#A),$
then $\#A\le f^{-1}(C/\delta)$ and $|n_j|\le  K_Z C/\delta$ for $j\in A$. Thus,  \eqref{E:4.3.1} and
 the volume argument, mentioned before the statement of our proposition, yields
\begin{align*}
 \vp^{-n}\le
\#\cF_{(\delta,C)}(x_i)\le\begin{pmatrix}N \\ \lfloor f^{-1}(C/\delta)\rfloor  \end{pmatrix}
 \Big(\frac{2K_Z C}{\delta} +1\Big)^{\lfloor f^{-1}(C/\delta)\rfloor}
 \le \Big[ N\Big(\frac{2K_Z C}{\delta} +1\Big)\Big]^{f^{-1}(C/\delta)},
\end{align*}
which, after taking $\ln(\cdot)$ on both sides, implies \eqref{E:4.3.3}.

Now assume that also \eqref{E:4.3.4} is satisfied. Let $(e_i,e^*_i)_{i=1}^n$ be an Auerbach basis of $X$,
i.e. $\|e_i\|=\|e_i^*\|=1$ and $e^*_i(e_j)=\delta_{(i,j)}$. Such a basis always exists
(c.f \cite[Theorem 5.6]{FHHMPZ}). Choose $0<\eta<\infty$ so that $\vp(1+1/\eta)<1$ and define for $i=1,2\ldots n$ 
\begin{align*}
A_i&=\Big\{j\in\{1,2\ldots N\}: |e^*_i(x_j)|\ge \vp/\eta K_ZCf^{-1}(C/\delta)n\Big\}
\end{align*}
Then  it follows for the right choice of $\sigma_j =\pm 1$, $j\in A_i$ that
\begin{align*}
g(\#A_i)&\ge  \Big\|\sum_{j\in A_i} \pm z_j\Big\|\\
         & \ge\frac1{\|S\|} \sup \Big\|\sum_{j\in A_i}  \pm x_j\Big\|\\
        & \ge \frac1{\|S\|}  e^*_i\Big(\!\sum_{j\in A_i^\sigma}  \sigma_jx_j\!\Big)
        \ge \frac{\#A_i\vp}{\eta\|S\|K_ZCf^{-1}(C/\delta)n}
\end{align*}
and thus
\begin{equation}\label{E:4.3.5} 
\frac{\#A_i}{g(\#A_i)} \le \frac{\eta\|S\|K_ZCf^{-1}(C/\delta)n}\vp.
\end{equation}
Put $A=\bigcup_{i\le n} A_i$. If $(n_j)_{j\le N}\subset \Z$ is such that $\sum_{j=1}^N \delta n_j x_j\in \cF_{(\delta,C)}(x_j)$, then
\begin{equation*}
\Big\|\sum_{j\in A^c} \delta n_j x_j\Big\|\le 
\sum_{j\in A^c,  n_j\not=0 }\Big(\sum_{i=1}^n |e^*_i(x_j)|\Big) \max_{j\le N} \delta| n_j| 
\le \#\{j: n_j\not=0\}   \frac{\vp}{\eta f^{-1}(C/\delta)} 
\le \frac{\vp}\eta, 
\end{equation*}
 where
  the second inequality follows from the definition of the $A_i$'s and the observation
  at the beginning of the proof, and
  the last inequality  follows from  the fact that
\begin{equation} \label{E:4.3.6} 
f(\#\{j: n_j\not=0\})\le \Big\|\sum n_j  z_j\Big\|\le C/\delta.
\end{equation}

This implies together with \eqref{E:4.3.1} that the set
$$\tilde \cF_{(\delta, C)}=\Big\{ \sum_{j\in A} n_j \delta x_j : (n_j)\subset \Z,\,
 \Big\|\sum_{j=1}^N n_j \delta z_j\Big\|\le C\Big\}$$
is $\vp\big(1+\frac1\eta\big)$-dense in  $B_X$. Hence,  our usual argument comparing volumes and \eqref{E:4.3.6}  yields
$$   \#A ^{f^{-1}(C/\delta)} \Big(\frac{2K_Z C}{\delta} +1\Big)^{f^{-1}(C/\delta) }\ge \begin{pmatrix} \#A \\ \lfloor f^{-1}(C/\delta)\rfloor \end{pmatrix}\Big(\frac{2K_Z C}{\delta} +1\Big)^{\lfloor f^{-1}(C/\delta)\rfloor } \ge  \frac1{\vp^n(1+1/\eta)^n} ,$$
Taking $\ln(\cdot)$  on both sides and letting $r(\ell)=\ln(\ell)g(\ell)/\ell$ for $\ell\in\N$, we conclude by \eqref{E:4.3.5} and since $\vp(1+1/\eta)<1$ that
\begin{align*}
n &\ln\Big(\frac{1}{\vp(1+1/\eta)}\Big)\\
&\le f^{-1}(C/\delta)\Big( \ln(\#A)+ \ln\Big(\frac{2K_Z C}{\delta} +1\Big)   \Big)\\
  &\le f^{-1}(C/\delta)\Big( \ln n +\max_{i\le n} \ln(\#A_i )+ \ln\Big(\frac{4K_Z C}{\delta} \Big)   \Big)\\
 &= f^{-1}(C/\delta)\Big( \ln n +\max_{i\le n} r(\#A_i )
     \Big(\frac{\#A_i}{g(\#A_i)}\Big) +\ln\Big(\frac{4K_Z C}{\delta} \Big) \Big)\\
    &\le f^{-1}(C/\delta)\Big( \ln n +n r(\#A_{i_0} ) \frac{\eta\|S\|K_Z Cf^{-1}(C/\delta)}{\vp}+\ln\Big(\frac{4K_Z C}{\delta} \Big) \Big)
\end{align*}
where $i_0\le n$ is chosen so that $\#A_{i_0}$ is maximal.
By our assumption on $g$ we can find an $\ell_0\in\N$ so that
$$f^{-1}(C/\delta)r(\ell)\le \frac12\ln\Big(\frac{1}{\vp(1+1/\eta)}\Big),\text{ whenever $\ell\ge \ell_0$}.$$
If $A_{i_0}\le \ell_0$ then 
$$n\ln\Big(\frac{1}{\vp(1+1/\eta)}\Big) \le f^{-1}(C/\delta) \Big[\ln n+\ln \ell_0 + \ln
\Big(\frac{2K_Z C}{\delta} +1\Big)\Big],$$
which implies that $n$ is bounded by a number which only depends on  $\vp$, $\delta$, $C$,
 $f$, $g$ and $K_Z$. If $\#A_{i_0}>\ell_0$, then it follows that 
$$\frac{n}{2}\ln\Big(\frac{1}{\vp(1+1/\eta)}\Big)\le \ f^{-1}(C/\delta)\Big( \ln n +
\frac{\eta\|S\|K_Z Cf^{-1}(C/\delta)}{\vp}+\ln\Big(\frac{4K_Z C}{\delta} \Big) \Big).$$
which implies our claim also in that case.
\end{proof}

We  shall formulate a corollary of Proposition \ref{P:4.3} for the infinite dimensional
situation.  We need first to introduce some notation and make some observations.

 Let $(x_i,f_i)_{i\in\N}$ be a frame of $X$. Furthermore assume that  $X$ has the  {\em $\pi_\lambda$-property}, which means that there is a sequence
  $\Pb=(P_n)$  of finite rank projections, whose norms are uniformly bounded, and 
 which {\em  approximate the identity}, i.e.
\begin{equation}\label{E:4.1}
  x=\lim_{n\to\infty} P_n(x), \text{ in norm  for all $x\in X$.}\end{equation}
For example, if $X$ has a basis $(e_i)$ we could choose for $n\in\N$ the projection onto the first $n$ coordinates, i.e.
$$P_n: X\to X,\quad \sum a_i e_i\mapsto \sum_{i=1}^n a_i e_i.$$

It is easy to see that
 $\big(P_n(x_i),f_i|_{P_n(X)}\big)$  
 is a frame of the space  $X_n=P_n(X)$. Moreover, condition 
 \eqref{E:4.1} and a straightforward compactness argument
 shows that for any $n\in\N$ and any $\frac12<r<1$ there is an $M_n=M(r,n)$ so that
it follows that
\begin{equation}\label{E:4.2} \Big\| x-\sum_{i=1}^N \la f_i,x\ra P_n(x_i )\Big\|\le (1-r)\|x\|,
\text{ whenever $x\in X_n$ and $N\ge M_n$}
\end{equation}

 It follows that the operators $(Q_n)$, with
$$Q_n: X_n\to X_n,\qquad x\mapsto \sum_{i=1}^{M_n}\la P_n^*(f_i),x\ra P_n(x_i ),$$
are uniformly bounded ($\|Q_n\|\le 2$, for $n\in\N$), invertible and their inverses are uniformly bounded
($\|Q_n^{-1}\|\le \frac1{r}$, for $n\in\N$). For $x\in X_n$ we write
\begin{equation*}
x= Q_n^{-1}Q_n(x)
=\sum_{i=1}^{M_n}\la P_n^*(f_i),x\ra (Q_n^{-1}\circ P_n)(x_i )
\end{equation*}
and deduce therefore that 
$$\big(y_i^{(n)},g_i^{(n)}\big)_{i=1}^{M_n}:=\big((Q_n^{-1}\circ P_n(x_i ),P_n^*(f_i)\big)_{i=1}^{M_n}$$
 is a finite frame of
$X_n$. 

Let $Z$ be a space with basis $(z_i)$ which is associated to the frame 
 $(x_i,f_i)$. 
It follows  easily that the operators $(S_n)$ and $(T_n)$
\begin{align*}&S_n:Z_n= [z_i:i\le M_n]\to X_n,\quad  z\mapsto \sum_{i=1}^{M_n} a_i z_i\mapsto
\sum_{i=1}^{M_n} a_i  (Q_n^{-1}\circ P_n)(x_i)\\
&T_n: X_n\to Z_n,\quad  x=\sum_{i=1}^{M_n} f_i(x) (Q_n^{-1}\circ P_n)(x_i)\mapsto \sum_{i=1}^{M_n} f_i(x) z_i\\
\end{align*}
are uniformly bounded, and thus $Z_n$ is an associated space
for the frame $(y^{(n)}_i,g_i^{(n)})_{i\le M_n}$ while $T_n$ and
$S_n$ are the associated decomposition and reconstruction operators, respectively.

Finally assume that the frame $(x_i,f_i)$ satisfies the $(\vp,\delta,C)$-NQP with respect to $Z$.
Again  by compactness   and using Proposition \ref{P:4.2} we can choose $M_n=M(r,n)$ large enough so that it also satisfies
\begin{align}\label{E:4.3}
&\text{For all $n\in\N$ and all $x\in X_n$ there is a sequence $(k_i)_{i=1}^{M_n}\subset\Z$ so that}\\
&\Big\|\sum_{i=1}^{M_n} k_i\delta z_i\Big\|\le C\|x\|\text{ and }\Big\|x- \sum_{i=1}^{M_n}\delta k_i x_i\Big\|\le\vp.\notag
\end{align}

 After changing $\vp>0$ and $\delta$ proportionally, if necessary, and since $r>\frac12$, we can assume that
$q=\frac{1-r}{r}+\sup_n{\|P_n\|}\frac{\vp}{r}<1$. 
For $n$ in $\N$ and $x\in B_{X_n}$ we can therefore  choose $(k_i)_{i=1}^{M_n}\subset \Z$ so that
 $\|\sum_{i=1}^{M_n} \delta k_i z_i\|\le C$ and
\begin{align*}
\Big\|x-\sum_{i=1}^{M_n} \delta k_i(Q_n^{-1}\circ P_n)(x_i)\Big\|&\le 
 \|Q^{-1}_n\|\cdot\Big\|Q_n(x)-\sum_{i=1}^{M_n} \delta k_iP_n(x_i)\Big\|\\
&\le\|Q^{-1}_n\|\cdot\|Q_n(x)- x\|+\|Q^{-1}_n\|\cdot\Big\|x-\sum_{i=1}^{M_n} \delta k_iP_n(x_i)\Big\|\\
&\le\|Q^{-1}_n\|\cdot\|Q_n(x)- x\|+\|Q^{-1}_n\|\cdot\|P_n\| \cdot \Big\|x-\sum_{i=1}^{M_n} \delta k_ix_i\Big\|\\
&\le \frac{1-r}{r}+\sup_n{\|P_n\|}\frac{\vp}{r}=q<1.
\end{align*}
Thus, for every $n\in\N$ the frame $\big(P^*_n(f_i),(Q_n^{-1}\circ P_n)(x_i)\big)_{i=1}^{M_n}$
satisfies condition \eqref{E:4.3.1} of Proposition \ref{P:4.3} (for $\vp=q$). Therefore we deduce the following Corollary.
\begin{cor}\label{C:4.4}
 Let $(x_i,f_i)_{i\in \N}$ be a frame of an infinite dimensional Banach space  $X$  for which there is a uniformly bounded sequence $(P_n)$ of finite rank projections which approximate the identity.
Assume that $(x_i,f_i)_{i\in \N}$ satisfies the 
 $(\vp,\delta,C)$-NQP with respect to a space $Z$ with basis $(z_i)$ for some choice of $\vp>0$, $\delta>0$ and $C$ 
so that  $q=\frac{1-r}{r}+\sup_n{\|P_n\|}\frac{2\vp}{r}<1$ with $\frac12<r<1$.
Let $(M_n)$ be any sequence in $\N$ which satisfies
\eqref{E:4.2} and \eqref{E:4.3}.

Finally assume that $(z_i)$ satisfies the following lower estimate 
$$\lim_{n\to\infty} \inf\Big\{ \Big\|\sum_{j\in A} n_j z_j\Big\|:
 (n_j)_{j\in A}\!\subset\! \Z\setminus\{0\}, A\subset \N, \#A=n\Big\}=\infty.$$

Then
\begin{enumerate}
\item[a)] $(M_n)$ increases exponentially with the dimension of $X_n$, i.e. there is a $c>1$, so that $M_n\ge c^{\dim(X_n)}$ eventually,
\item[b)] $\displaystyle \limsup_{n\to\infty} \frac {\ln(n)}{n}\sup\Big\{\Big\|\sum_{i\in A} \pm z_i \Big\|: A\subset \N,
 \#A=n,\Big\}=\infty$.
\end{enumerate}
\end{cor}
Let us simplify the conditions in Corollary \ref{C:4.4} and observe that it implies the following.
\begin{cor}\label{C:4.5}
Assume that $X$ is an infinite dimensional Banach space with the $\pi_\lambda$-property and
that $Z$ is a Banach space with a  basis $(z_i)$ satisfying for some choice of $1<q<p<\infty$
{\em lower $\ell_p$ and upper $\ell_q$ estimates}, which means that for some $C$
$$\frac1C\Big(\sum |a_i|^p\Big)^{1/p}\le \Big\|\sum a_i z_i \Big\|\le
 C\Big(\sum |a_i|^q\Big)^{1/q}.$$

Then no frame of  $X$ has the $NQP$ with respect to $\big(Z, (z_i)\big)$. 
\end{cor}

The following example shows how to construct a frame with respect to a space $Z$ which contains $\ell_1$.

\begin{prop}\label{P:4.6}
Let $(x_i,f_i)_{i\in\N}$ be any frame of a Banach space $X$ and let $Z$ be a space with semi-normalized basis
 $(z_i)$, which is associated to $(x_i,f_i)$.
Then there is a  frame $(\tilde x_i, \tilde f_i)_{i\in\N}$  and a basis $(\tilde z_i)$ of
 $\tilde Z=Z\oplus_\infty\ell_1$ so that $(\tilde x_i, \tilde f_i)_{i\in\N}$ has
$\tilde Z$ as an associated space and has the NQP with respect to $\tilde Z$. Moreover,
 $(\tilde x_i, \tilde f_i)_{i\in\N}$ is semi-normalized if 
 $(x_i,f_i)_{i\in\N}$  has this property (for example if $(x_i)$ is a normalized basis of $X$).  

\end{prop}
\begin{proof}
 Assume, without loss of generality that  $\|z_i\|=1$ for $i\in\N$.
Choose a quotient map $Q:\ell_1\to X$ so that $(Q(e_i):i\in\N)$ is a $\frac12$-net in $B_X$
and so that $\big\|Q(e_i)\pm x_i\big\|>\frac14$ for $i\in\N$ (which is easy to accomplish). 
Finally we apply Proposition \ref{P:2.7} to $Y=\ell_1$ with its unit vector basis $(e_i)$
 and $V=Q$, and  observe
  that the frame $(\xt_i,\tilde f_i)$ and basis $(\zt_i)$ of $\tilde Z$, as  constructed there, has the property that
for any $x\in B_X$ there is an $i\in\N$ so that
$$\Big\|x-\frac{\tilde x_{2i}-\tilde x_{2i-1}}2\Big\|= \|x-Q(e_i)\|\le \frac12 \text{ and }
 \Big\|\frac12  (\tilde z_{2i}-\tilde z_{2i-1})\Big\|= 1$$
which implies by Proposition \ref{P:4.2} that $(x_i,f_i)_{i\in\N}$ has the NQP with respect to $(z_i)$.

By construction of $(\xt_i,\ft_i)$ in Proposition \ref{P:2.7} it follows that $(\ft_i)$ is semi-normalized if $(f_i)$ has this property and since $\big\|Q(e_i)\pm x_i\big\|>\frac14$, for $i\in \N$, it follows that
$$\frac14\le \|\xt_i\|\le \sup_{j\in\N}\|x_j\| +1,$$
which implies that $(\tilde x_i, \tilde f_i)$ is semi-normalized if  $(x_i,f_i)$ has this property. 
\end{proof}

Finally let us present an infinite  dimensional argument implying that if $Z$ is a reflexive
space with basis it cannot be the associated space  of a frame $(x_i,f_i)_{i\in\N}$, with $\|x_i\|=1$, for $i\in\N$, which satisfies the NQP.

\begin{prop}\label{P:4.7} Assume that $Z$ is a reflexive space with
  a semi-normalized basis $(z_i)$, and assume that $(x_i,(f_i))$ is a frame of an infinite dimensional Banach space $X$ with associated space $Z$.

Then $((x_i),(f_i))$  cannot have the NQP with respect to $Z$.
\end{prop}
The following result follows from Proposition \ref{P:4.7} as well as from 
Corollary \ref{C:4.5}.
\begin{cor} A semi-normalized frame of an infinite dimensional 
 Hilbert space H   cannot have the NQP with respect to  the associated Hilbert space $\ell_2(\N)$.
\end{cor}
\begin{proof}[Proof of Proposition \ref{P:4.7}]
We assume w.l.o.g. that $(z_i)$ is bimonotone and let $T:X\to Z$ and $S:Z\to X$
 be the associated decomposition and reconstruction operator, respectively.

For $C<\infty$ and $\delta>0$ define 
$$\cB_{(C,\delta)}=\Big\{\sum_{i=1}^\infty \delta k_i z_i\in Z: (k_i)\subset \Z,
\|\sum_{i=1}^\infty \delta k_i z_i\|\le C\Big\}.$$ 
Assume that $(x_i,f_i)$ has the NQP with respect to $Z$. Then we can choose
$\delta>0$ small enough and $C\ge 1$ large enough so that $S(B_{(C,\delta)})$ is $\vp$-dense
in $B_X$ for some $0<\vp<1$.

Since $(z_i)$ is semi-normalized and $Z$ is
 reflexive, $\cB_{(C,\delta)}$ is weakly compact. Indeed, assume
that for $n\in\N$,
$$y_n=\sum_{i=1}^\infty \delta k_i^{(n)} z_i\in \cB_{(C,\delta)}.$$
After passing to  subsequence we can assume that for all
 $i\in\N$ there is a $k_i\in\N$ so that
  $k^{(n)}_i=k_i$ whenever $n\ge i$. Thus, by bimonotonicity,
  it follows that $\|\sum_{i=1}^n \delta k_iz_i\|\le C$, for all $n\in\N$,
  and, thus,  since $(z_i)$ is boundedly complete
   $\sum_{i=1}^\infty \delta k_iz_i\in Z$ and 
$\|\sum_{i=1}^\infty \delta k_iz_i\|\le C$. Thus,
 $\sum_{i=1}^\infty \delta k_iz_i$ in $B_{(C,\delta)}$,
 and since $k_i^{(n)}$ converges point-wise to $(k_i)$ and $(z_i)$ is shrinking it is the weak limit
 of $y^{(n)}$. The support of each element in $B_{(C,\delta)}$ is finite
  since $(z_i)$ is a semi-normalized basis, and thus $B_{(C,\delta)}$ is 
 countable. Since $S(B_{(C,\delta)})$ is $\vp$-dense in $X$, it follows that
the map
$$ E: X^*\to C(\cB_{(C,\delta)}), \text{ with }
  E(x^*)\Big(\sum \delta k_i z_i\Big)=\sum \delta k_i S^*(x^*)(z_i),$$
is an isomorphic embedding.
Indeed for $x^*\in B_X^*$ there is an $x\in B_X$ so that $|x^*(x)|=1$ and a sequence $(k_i)\in \Z\cap \coo$ so that
 $\|x-\sum \delta k_i x_i\|\le \vp$, and thus
$$\|E(x^*)\|\ge \Big| E(x^*)\Big(\sum \delta k_i z_i\Big)\Big|=x^*\Big(\sum \delta k_i x_i\Big)= 1+
x^*\Big(\sum \delta k_i x_i-x\Big)\ge 1-\vp.$$
But this  would means that $X^*$ is isomorphic to a subspace of the space of continuous functions on a countable compact space,
 and, thus, hereditarily $c_0$, which is impossible since $X$ is a quotient of a reflexive space and thus 
also reflexive.\end{proof}

\section{Quantization and Cotype}\label{S:5}

In this section we consider a quantization concept for Schauder frames, which is independent of an associated space.

\begin{defin}\label{D:5.1}  Let $(x_i,f_i)_{i\in\J}$ be a frame of a (finite or infinite dimensional) Banach space $X$, $\J=\N$ or $\J=\{1,2,\ldots  N\}$, for some $N\in\N$, and let  $0<\vp$,  $0<\delta\le1$ and $1\le C<\infty$. We say that
  $(x_i,f_i)_{i\in\J}$  satisfies the {\em $(\vp,\delta,C)$-Bounded Coefficient Net Quantization Property} or  
$(\vp,\delta,C)$-BCNQP if for all $(a_i)_{i\in\J}\in[-1,1]^\J\cap \coo(\J)$ there is a  $(k_i)_{i\in\J}\in \Z^{\J}\cap \coo(\J)$ so that
$$\Big\|\sum_{i\in\J} a_i x_i -\sum_{i\in\J}\delta k_ix_i\Big\|\le \vp \text{ and } \max_{i\in\J} |k_i|\le \frac{C}{\delta}.$$

\end{defin}

\begin{remark}\label{R:5.2} 
 Let $(x_i,f_i)_{i\in\J}$ be a frame of $X$ and let $0<\vp$, $0<\delta\le1$ and $1\le C<\infty$.
 
\begin{enumerate}
\item[a)]
Since for any $(a_i)_{i\in\J}\in \coo(\J)$ and any $i\in\J$ we can write
 $a_i=m_i\delta+\tilde a_i$ with $m_i\in\N$, $|m_i|\delta\le |a_i|$ and 
 $|\tilde a_i|\le \delta$, for $i\in\J$,  $(x_i,f_i)$ satisfies the $(\vp,\delta,C)$-BCNQP 
implies that 
\begin{align}\label{E:5.2.1} &\text{for all $(a_i)_{i\in\J}\in \coo(\J)$ there is a  $(k_i)_{i\in\J}\in \Z^{\J}\cap \coo(\J)$ so that}\qquad\qquad\qquad\\
&\qquad\Big\|\sum_{i\in\J} a_i x_i -\sum_{i\in\J}\delta k_ix_i\Big\|\le \vp \text{ and } \max_{i\in\J}  |k_i|\le \max_{i\in\J} |a_i|+\frac{C}{\delta}.\notag\end{align}

(a) immediately    implies 

\item[b)] If $(x_i,f_i)$ satisfies $(\vp,\delta,C)$-BCNQP and $0<\lambda \le 1$ then $(x_i,f_i)$ satisfies 
  $(\lambda\vp,\lambda\delta,1+\lambda C)$-BCNQP.
\item[c)] If $(x_i)$ is a semi-normalized basis of $X$ and $(f_i)$ are  the coordinate functionals  with respect to  $(x_i)$ and $(x_i)$ satisfies the 
$(\vp,\delta)$-NQP (Definition \ref{D:3.1}), then $(x_i,f_i)$ satisfies the    $(\vp,\delta,C)$-BCNQP  with 
$C=1+\vp\sup_{i\in\J}\|f_i\|$. Indeed, for $x\in X$, $x=\sum_{i=1} a_i x_i$, with $|a_i|\le 1$,
  there is a sequence $(k_i)\in \Z$ with finite support so that $\big\|x-\sum \delta k_i x_i\big\|\le \vp$ and
$$\delta\max |k_i|\le \max_i \Big(|f_i(x)|   + \Big|f_i\Big(x-\sum \delta k_i x_i\Big)\Big|\Big)\le 1+\vp\sup_{i\in\J}\|f_i\|.$$

\end{enumerate}
\end{remark}

We will connect the property BCNQP with properties of the {\em cotype } of the Banach space. 

\begin{defin}\label{D:5.3}
Let $p\le 2$.   We say that a Banach space $X$ {\em has type $p$} if there is a $c<\infty $ so that for all $n\in\N$ and all vectors $x_1,x_2,\ldots x_n\in X$ 
$$\Big(\text{ave}\Big\| \sum_{i=1}^n\pm x_i\Big\|^2\Big)^{1/2}
=
  \Big( 2^{-n}\sum_{(\sigma_i)_{i=1}^n\in\{\pm1\}^n} \Big\|\sum_{i=1}^n\sigma_i x_i\Big\|^2\Big)^{1/2}
\le c \Big(\sum_{i=1}^n \|x_i\|^p\Big)^{1/p}.$$
In that case the smallest such $c$ will be denoted by $T_p(X)$.

Let $q\ge 2$. We say that a Banach space $X$ {\em has cotype $q$} if there is a $c<\infty $ so that
for all $n\in\N$ and all $x_1,x_2,\ldots x_n\in X$:
$$\Big(\sum_{i=1}^n \|x_i\|^q\Big)^{1/q}\le c \Big(\text{ave}\Big\| \sum_{i=1}^n\pm x_i\Big\|^2\Big)^{1/2}
=
  c\Big( 2^{-n}\sum_{(\sigma_i)_{i=1}^n\in\{\pm1\}^n} \Big\|\sum_{i=1}^n\sigma_i x_i\Big\|^2\Big)^{1/2}.
$$
The smallest of all these constants will be denoted by $C_q(X)$.

We say that $X$ has only trivial type, or only  trivial cotype if $T_P(X)=\infty$ for all $p>1$, or
 $C_q(X)=\infty$, for all $q<\infty$. 
\end{defin}

Basic properties of spaces with type and cotype can be found for example in \cite{DJT} or \cite{Pi2}. We are mainly interested in estimates 
 of the volume ratio of  the unit ball $B_X$ of a finite dimensional space $X$ 
  using $C_q(X)$ and  the connection between finite cotype and the lack of containing $\ell_\infty^n$'s uniformly. 

Assume 
$X$ is an $n$-dimensional space which we identify with $(\R^n, \|\cdot\|)$.
 Let $E$ be the John ellipsoid of the unit ball $B_X$ of $X$,  i.e. 
the  ellipsoid contained in $B_X$ having maximal volume.
 It was show in  \cite{Jo} (see also \cite[Chapter 3]{Pi2}) that $E$ is unique. 
 We call the ratio $\Vol^{1/n}(B_X)/\Vol^{1/n}(E)$ the {\em  volume ratio of $B_X$}.
 Combining
\cite[Theorem 6]{Ro}, which establishes an upper estimate for the volume ratio using  $T_p(X^*)$, with 
 a result of  
 Maurey and  Pisier \cite{MP1,MP2} (see also \cite[Proposition 13.17]{DJT}) estimating $T_p(X^*)$ and  a result
   of Pisier (\cite{Pi1} (see also \cite[Theorem 2.5] {Pi2}) estimating the $K$-convexity constant $K(X)$ of
 $X$,   we obtain the  connection between the 
 volume ratio of $B_X$ and the
 cotype constant of $X$.
\begin{thm}\label{T:5.3}
There is a  universal constant $d$ so that for all finite dimensional Banach spaces $X$, with $n=\dim(X)\ge 2$, and all $2\le q<\infty$.
 \begin{equation}\label{E:5.3.1}
\left(\frac{\Vol(B_X)}{\Vol(E)} \right)^{1/n}\le d C_q(X) n^{\alpha(q)}\ln n,
\end{equation}
where $E\subset X$ is the John ellipsoid of $B_X$ and $\alpha(q):= \frac12 -\frac1q$.
\end{thm}
We will also need a second upper estimate for the volume ratio due to Milman and Pisier \cite{MiP}.
\begin{thm}\label{T:5.3a}{\rm\cite{MiP}}(see also \cite[Theorem 10.4] {Pi2})
There is a  universal  constant $A$ so that for any finite dimensional Banach space $X$,
 \begin{equation}\label{E:5.3a.1}
\left(\frac{\Vol(B_X)}{\Vol(E)} \right)^{1/n}\le g(C_2(X)):= A C_2(X) \ln (1+C_2(X))
\end{equation}where $E\subset X$ is the John ellipsoid of $B_X$.
\end{thm}

The next result describes the  connection between the property of having a finite cotype for $q<\infty$ and the lack of of containing $\ell_\infty^n$'s uniformly.

\begin{thm}\label{T:5.4}{\rm \cite{MP1,MP2}} For $N\in \N$  there is 
  a $q(N)\in(2,\infty)$ and a $C(N)<\infty$ so that:
  \begin{align}\label{E:6.1.5}
&\text{For  any (finite or infinite dimensional) Banach space $X$   which does not 
}\\ 
&\text{contain a 2-isomorphic copy of $\ell_\infty^N$  we have  that }C_{q(N)}(X)\le C(N).\notag
\end{align}
\end{thm}

Finally we will need the following 
 result from \cite{Os}.  It is implicitly already
contained in  \cite[pp.95--97]{GMP}, and it has  probably been  known for much longer.

In order to state it we will need the following notation. Let $m\le n\in\N$ and let $L\subset  \R^n$ be an $m$-dimensional
subspace. Let $Q_n$ be the unit cube in $\R^n$ .  By a simple compactness argument there is a projection $P:\R^n\to L$ for which
$\Vol(P(Q_n))$ is minimal. In that case we call the image  $P(Q_n)$ a {\em minimal-volume projection of $Q_n$ onto $L$}. 

\begin{thm}\label{T:5.5} \cite[Theorem 1]{Os} Let $L$ be a linear subspace of $\R^n$, and let $\cM$ be the set of
all minimal volume projections of $Q_n$ onto $L$. 

Then $\cM$ contains a parallelepiped.
\end{thm}

We are now in the position to state and to prove the connection
between cotype and BCNQP in the finite dimensional case.

\begin{thm}\label{T:5.6} 
There is a map
$n_0: [1,\infty)^2\to [0,\infty)$
so that for   all finite dimensional Banach spaces  $X$ the following holds. 

If  $(x_i,f_i)_{i=1}^N$  is a frame of $X$, with $\|x_i\|=1$, for $i\le N$, and  which  satisfies $(1,\delta,C)$-BCNQP for some $0<\delta<1$ and $C\ge 1$,
then  for all $2\le q<\infty$
$$N\ge \dim(X)\ln(\dim(X)) \frac{1}{2q\ln \big(1+2\frac{C}\delta\big)} ,\text{ whenever 
$\dim(X)\ge n_0\Big(\frac C\delta, K C_q(X)\Big)$},$$
where $K$ is the projection constant of $(x_i,f_i)_{i=1}^N$.
\end{thm}

\begin{proof} Let $Z$ be the space with a basis $(z_i)$ and let $T: X\to Z$ and   $S:Z\to X$ be the associated  decomposition and the reconstruction operator as constructed in  the proof of Proposition \ref{P:2.4} (a) ``$\Rightarrow$''.
Since the $x_i$'s are normalized  the $z_i$'s are also of norm 1.
 After a linear  transformation we can assume that  $Z=\R^N$ and  
$(z_i)_{i=1}^N$ is the
 unit vector basis of $\R^N$.  
 Since $(z_i)$ is a bimonotone basis it follows that $\|z^*_i\|=1$, for $i\le N$ .
 Hence 
 $B_Z\subset  Q_N$, where $Q_N$ denotes the unit cube in $\R^N$.
Define
$L=T(X)$  and put $n=\dim(X)=\dim(L)$. Since $S\circ T=Id_X$  it follows
that $P=T\circ S$ is a projection from $Z$ onto $L$ and if we denote the John ellipsoid of $T(B_X)$
 by $E$ and we deduce that (recall that by Proposition \ref{P:2.4} (a) $\|T\|\le K$)
 \begin{equation}\label{E:5.6.1} 
 E\subset  T(B_X) =P\circ T(B_X)\subset  P(K\cdot B_Z)\subset P(K\cdot Q_N).
\end{equation}
 By Theorem \ref{T:5.5} there  is a minimal -volume projection $M$ of $Q_N$ onto $L$ which
 is a parallelepiped. Let $B_n$ denote the $n$ dimensional Euclidean ball in $\R^n$.
 Since there is a universal constant $c$ so that
 $$\Vol(B_n)\le  \left(\frac{c}{\sqrt n}\right)^n,$$
and since
$$\frac1{K}E\subset \frac1{\|T\|}E\subset L\cap B_Z\subset L\cap Q_N\subset  M,$$
we deduce from the fact that $B_n$ is the John ellipsoid of  the unit cube in $\R^n$ \cite{Jo}
(see also \cite[Chapter 3]{Pi2}), that
$$K\left( \frac{\Vol(P(Q_N))}{\Vol\big( E\big)}\right)^{1/n}
=\left( \frac{\Vol(P( Q_N))}{\Vol\big( \frac1{K}E\big)}\right)^{1/n} \ge
\left( \frac{\Vol( M)}{\Vol\big( \frac1{K}E\big)}\right)^{1/n}\ge \frac{\sqrt{n}}{c}
$$
The last inequality follows from applying a linear transformation $A$ to $L$ so that
$A(M)$ is a the unit cube in $L$ (with respect to some orthonormal basis of $L$) and, thus $A(\frac1K E)$ is an ellipsoid whose volume cannot exceed that
 of the Euclidean unit ball in $L$.
Since $T:(X,\|\cdot\|)\to (L,\|\cdot\|_{T(B_X)})$, where $\|\cdot\|_{T(B_X)}$ is the Minkowski
 functional for ${T(B_X)}$, is an isometry it follows from 
 Theorem \ref{T:5.3}  that
\begin{align}\label{E:5.6.2} \Vol^{1/n} (T(B_X))&\le d C_q(X) n^{\alpha(q)} \ln(n) \Vol^{1/n}(E)\\
 &\le  d cK C_q(X) n^{-\frac1q} \ln(n) \Vol^{1/n}(P(Q_N))\notag
\end{align}
(the universal constant  $d$ was introduced in Theorem \ref{T:5.3}).

Since the zonotope
$$P(Q_N)= T\circ S \Big(\Big\{\sum_{i=1}^N a_i z_i:  |a_i|\le 1 \Big\}\Big)=\Big\{ \sum_{i=1}^N a_i T(x_i):  |a_i|\le 1\Big\},$$
contains at most $\big(1+\frac{2C}{\delta})^N$ points from  the set
$D=\{\sum \delta n_i T(x_i) : (n_i)\subset\Z, \max \delta |n_i|\!\le\!C\}$ and since
  from our assumption  that $(x_if_i)_{i=1}^N$ satisfies the $(1,\delta,C)$-BCNQP it follows that
 $$P(Q_N)\subset \bigcup_{z\in D} z+ T(B_X),$$
we deduce that
$$\Big(1+\frac{2C}{\delta}\Big)^{N}\ge \frac{\Vol(P(Q_N))}{\Vol(T(B_X))}
\ge \left( \frac{n^{1/q}}{  Kdc  C_q(X)\ln (n)}\right)^n $$
and, thus,
$$N\ge \frac{  n\ln(n)}{q\ln \big(1+\frac{C}\delta\big)} -
\frac{n\ln(\ln(n)) }{\ln \big(1+\frac{2C}\delta\big)}-
 \frac{n\ln(dc K  C_q(X))}{\ln \big(1+\frac{2C}\delta\big)},$$
which easily implies our claim.
\end{proof}

In the  next result we will show that,  up to a  constant factor,  the result in 
Theorem \ref{T:5.6} is sharp.
We are using the simple fact that for any number $0\le r\le1$  and any $m\in\N$, 
$r$ can be approximated by a finite sum of dyadic numbers, say $\tilde r=\sum_{j=1}^m \sigma_j 2^{-j}$,
$\sigma_j\in\{0,1\}$, for $j=1,\ldots m$, so that $|r-\tilde r|\le 2^{-m}$.

\begin{prop}\label{P:5.7}
Let $X$ be an $n$-dimensional space with an Auerbach basis $(e_i,e^*_i)_{i=1}^n$. Let $m\in \N$ and let
 $K$ be the projection constant of $(e_i)_{i=1}^n$.
Then there is a frame $(x_{(i,j,s)},f_{(i,j,s)}:1\le i\le n,1\le  j\le m, s=0,1)$ (ordered lexicographically) so that
\begin{align}\label{E:5.7.1}
&\frac12 \le \|x_{(i,j,s)}\|\le 2\text{ and } \|f_{(i,j,s)}\|=1,
\text {whenever }1\!\le\! i\!\le\! n, 1\!\le\!j\!\le\! m\text{ and } s\!=\!0,1,\\
&\label{E:5.7.2}
\forall  (a_{(i,j,s)}:\!1\le i\!\le\!n,\,1\le \!j\!\le\! m,\!s\!=\!0,\!1)\!\subset\![-1,1]\,\\
 &\qquad\qquad \exists (k_{(i,j,s)}\!:\!i\le\!n,\! j\!\le\!m,\! s\!=\!0,\!1)\!\subset\!
\{-3,\!-2,...,3\}\notag \\
&\Big\|\sum_{i=1}^n\sum_{j=1}^m\sum_{s=0}^1 a_{(i,j,s)} x_{(i,j,s)}-
\sum_{i=1}^n\sum_{j=1}^m\sum_{s=0}^1 k_{(i,j,s)}
 x_{(i,j,s)}\Big\| \le 1+n\frac{2^{-m}}{1-2^{-m}}\notag\\
&\big(\text{i.e. $(x_{(i,j,s)},f_{(i,j,s)}:1\le i\!\le\!n,\,1\le j\!\le\!m, s\!=\!0,\!1)$ satisfies the}\notag\\
&\text{$\big(1+n\frac{2^{-m}}{1\!-\!2^{-m}},1,3\big)$-BCNQP}\big).\notag\\
&\label{E:5.7.3}\text{The projection constant of $(x_{(i,j,s)},f_{(i,j,s)}:1\!\le\!i\!\le\!n, 1\!\le\!j\!\le\!m, s\!=\!0,\!1)$ does not}\\
&\text{exceed $4K$.}\notag
\end{align}
\end{prop}

\begin{proof} For $1\le i\le n$ and  $1\le j\le m$ define
$
x_{(i,j,0)}=e_1$, $x_{(i,j,1)}=e_1+\frac{2^{-j}}{1-2^{-m}} e_i$, 
$f_{(i,j,0)}=-e^{*}_i$ and $f_{(i,j,1)}=e^*_i$.
Since for every $x\in X$
\begin{align*}
x=\sum_{i=1}^n e^*_i(x) e_i
  =\sum_{i=1}^n \sum_{j=1}^me^*_i(x) e_i\frac{2^{-j}}{1-2^{-m}}
  = \sum_{i=1}^n \sum_{j=1}^m f_{(i,j,1)}(x)  x_{(i,j,1)}+f_{(i,j,0)}(x)x_{(i,j,0)},
  \end{align*}  
  $(x_{(i,j,s)},f_{(i,j,s)}:1\!\le\!i\le\!n,1\le\!j\!\le m, s=0,1)$ is a frame of $X$ and it satisfies
  \eqref{E:5.7.1}.
In order to verify \eqref{E:5.7.2}  let $(a_{(i,j,s)}:1\!\le\!i\le\!n, 1\!\le\!j\!\le\!m, s=0,1)\subset[-1,1]$ be given.
  For $i=1,2,\ldots, n$ it follows that $\Big|\sum_{j=1}^m a_{(i,j,1)} \frac{2^{-j}}{1-2^{-m}}\Big|\le 1$, and, thus, we can choose $(k_{(i,j,1)}:j\le m)\subset\{0,\pm1\}$ so that for each $i\le n$
  \begin{equation}\label{E:5.7.4}
  \Big|\sum_{j=1}^m a_{(i,j,1)} \frac{2^{-j}}{1-2^{-m} }-\sum_{j=1}^m k_{(i,j,1)} 
  \frac{2^{-j}}{1-2^{-m} }\Big|\le \frac{2^{-m}}{1-2^{-m}}.
  \end{equation} 
  Since  the absolute value of
  $M=\sum_{i=1}^n\sum_{j=1}^m a(i,j,1)+a(i,j,0)-k_{(i,j,1)}$
  is  at most $3nm$ we can choose for $1\le i\le n$ and $1\le j\le m$,  $k_{(i,j,0)}\in\{-3,-2,\ldots, 2,3\}$ so that 
   $a=M-\sum_{i=1}^n\sum_{j=1}^mk_{(i,j,0)}$,
 has absolute value at most $1$. We compute
  \begin{align*}
   \Big\|\sum_{i=1}^n&\sum_{j=1}^m\sum_{s=0}^1 a_{(i,j,s)} x_{(i,j,s)}-
\sum_{i=1}^n\sum_{j=1}^m\sum_{s=0}^1 k_{(i,j,s)} x_{(i,j,s)}\Big\|\\
 &\le   \Big\|\sum_{i=1}^n\sum_{j=1}^m a_{(i,j,1)} e_i\frac{2^{-j}}{1-2^{-m}}-
   \sum_{i=1}^n\sum_{j=1}^m k_{(i,j,1)} e_i\frac{2^{-j}}{1-2^{-m}}\Big\|\\
   &\qquad+ \Big|\sum_{i=1}^n\sum_{j=1}^m a_{(i,j,1)} +a_{(i,j,0)}-k_{(i,j,0)}-k_{(i,j,1)} \Big|\\
    &\le  1+\sum_{i=1}^n \Big|\sum_{j=1}^m (a_{(i,j,1)}-k_{(i,j,1)}) \frac{2^{-j}}{1-2^{-m}}\Big|
\le 1+ n\frac{2^{-m}}{1-2^{-m}}
  \end{align*}
  which proves \eqref{E:5.7.2}.
  
  To estimate the projection constant of 
  $(x_{(i,j,s)},f_{(i,j,s)}:1\!\le\!i\le n,1\le\! j\!\le m, s\!=\!0,1)$  we denote by $\le_{\text{lex}}$  the lexicographic order on 
  $\{(i,j,s) :i\le n, j\le m, s=0,1\}$, and let
  $$x=\sum_{i=1}^n a_ie_i =\sum_{i=1}^n\sum_{j=1}^m -a_ie_1+a_ie_1
  +a_ie_i\frac{2^{-j}}{1-2^-m}
   =\sum_{i=1}^n\sum_{j=1}^m\sum_{s=0,1} f_{(i,j,s)}(x)x_{(i,j,s)}
   $$
  and $(i_0,j_0,s_0)\le_{\text{lex}}(i_1,j_1,s_1)$. Then, if $i_0<i_1$,
  \begin{align*}
 \Big\| &\sum_{(i_0,j_0,s_0)\le_{\text{lex}}(i,j,s)\le_{\text{lex}}(i_1,j_1,s_1) }
 f_{(i,j,s)}(x )x_{(i,j,s)}\Big\|\\
 &=\Big\| 1_{\{s_0=1\}}\Big[a_{i_0}e_1
  +a_{i_0}\frac{2^{-j_0}}{1-2^{-m}}\Big]
 +\sum_{j=j_0+1}^m -a_{i_0} e_1 +a_{i_0} e_1 + a_{i_0}
  \frac{2^{-j}}{1-2^{-m}}e_{i_0}\\
&\qquad+\sum_{i=i_0+1}^{i_1-1}\sum_{j=1}^m \Big(-a_ie_1+a_ie_1+
   a_{i} \frac{2^{-j}}{1-2^{-m}}e_{i}\Big)\\
  &\qquad+\sum_{j=1}^{j_1} \Big(-a_{i_1} e_{1}+a_{i_1}  e_{1}+    
  a_{i_1} \frac{2^{-j}}{1-2^{-m}}e_{i_1}\Big) - 1_{\{s_1=0\}}
  a_{i_1}  e_{1} \Big\|\\
   &\le 2|a_{i_0}|+\Big\|\sum_{i=i_0+1}^{i_1-1} a_i e_i\Big\|+|a_{i_1}|\le 4K\|x\|.
  \end{align*} 
 If $i_0=i_1$  similar estimates give the to the same result  for the remaining cases
 and \eqref{E:5.7.3} follows.
\end{proof}

\begin{remark}\label{R:5.6b} If we choose in Proposition  \ref{P:5.7} $m=\lceil 2\log n\rceil$ and thus $2^{-m}\simeq 1/n^2$ we obtain a frame for $X$ of  approximate size   $4n\log_2(n)$ 
having the $(3,1,3)$-BCNQP. Thus as we mentioned earlier, up to a constant Theorem \ref{T:5.6} is best possible.
\end{remark}

\begin{remark}\label{R:5.6a} In Theorem \ref{T:5.6} we assumed for simplicity that the $x_i$'s of our frame are normalized.
It is easy to see that the same proof works for  a general frame, in that case
$n_0$ depends also on $a=\min\{\|x_i\|: i\le N, x_i\not=0\}$ and $b=\max \{\|x_i\|:i\le N\}$.
\end{remark}

With  a similar proof to that  of Theorem \ref{T:5.6} we derive an upper estimate  for 
  $\min_{i\le N}\|x_i\|$, $i\le N$,
 assuming  that $(x_i,f_i)_{i=1}^N$ is a frame of an $n$ dimensional space $X$ which satisfies
 the $(1,\delta,C)$-BCNQP for some choice of $\delta>0$ and $C<\infty$ assuming  that $N$ is proportional to $n$.

\begin{thm}\label{T:5.6a} For any choice of $\delta\in(0,1]$, and $C,K,q, c_2 \ge 1$ there is
a value $h=h(\delta,C,K,q, c_2)$ so that the following holds for all $n\in\N$.

If $X$ is an $n$-dimensional space, $N\le q n$ and $(x_i,f_i)_{i=1}^N$ is a frame of $X$  with projection constant $K$
which has the $(1,\delta,C)$-BCNQP, then if $C_2(X)\le c_2$,
$$ \min_{i\le N}\|x_i\|\le  \frac{h(\delta,C,K,q, c_2)}{\sqrt{n}}. $$
\end{thm}

\begin{proof}[Sketch of proof] Let  $(x_i,f_i)_{i=1}^N$ be a frame of $X$, $N\le qn$, which has the  $(1,\delta,C)$-BCNQP and projection constant $K$.
As in the proof of  Theorem \ref{T:5.6} we let $Z$ be the  associated space with basis $(z_i)$ which was constructed in Proposition \ref{P:2.4}, $T:X\to Z$ the  associated decomposition operator, and $S$ the associated  reconstruction operator. Let
 $L=T(X)$, and $P=T\circ S$,
 and let us  also assume that $Z=\R^N$  and $z_i=e_i$ for $i\le N$. Note that now
  $\|z_i\|=\|x_i\|$ and $z^*_i=\|x_i\|^{-1}$ and we can therefore follow the proof
 of  Theorem \ref{T:5.6}  replacing $Q_N$ by  the box
$$\tilde Q_N=\prod_{i=1}^N\Big[-\frac1{\|x_i\|}, \frac1{\|x_i\|}\Big].$$
As in the proof  of Theorem  \ref{T:5.6} it follows that
$\frac1K T(B_X)\subset P(B_Z)\subset  P(\tilde Q_N)$. 
For  the John ellipsoid $E$ of $T(B_X)$ it follows therefore that $\frac1K E\subset M$, where a $M$ is a minimal volume 
projection of $\tilde Q_N$ which is also a parallelepiped  in $L$, and as before we deduce  that
 $K\Vol^{1/n}(P(\tilde Q_N))\ge \Vol^{1/n}(E) \sqrt{n}/c$.
Instead of applying Theorem \ref{T:5.3} we now use  Theorem \ref{T:5.3a} and letting
 $\alpha =\min_{i\le N}\|x_i\|$ we  deduce that
\begin{align*}
\Vol^{1/n}(T(B_X))&\le g(C_2(X)) \Vol^{1/n}(E)\\
 &\le \frac{g(C_2(X))cK\Vol^{1/n}(P(\tilde Q_N))}{\sqrt{n}}\le
 \frac{g(C_2(X)) c K\Vol^{1/n}(P(Q_N))}{\sqrt{n} \alpha}.
\end{align*}
We can again compare the volume of the zonotope $P(Q_N)$ with the  volume of the union
$\bigcup_{z\in D} z+T(B_X)$, where $D$ is defined as in the proof of Theorem \ref{T:5.3}, and deduce that
$$\Big(1+\frac{2C}\delta\Big)^{qn}\ge
\Big(1+\frac{2C}\delta\Big)^N\ge \frac{\Vol(P(Q_N))}{\Vol(T(B_X))}\ge 
\Big(\frac{\sqrt{n} \alpha}{g(C_2(X))cK}\Big)^n.
 $$
Taking the $n$th root on both sides yields our claim.
\end{proof}

\begin{remark}\label{R:5.8a}
In section \ref{S:6} we will recall a result of Lyubarski and Vershinin \cite{LV} which shows that
  for $q>1$ there are $\vp<1$ , $\delta<1$, $C<\infty$  so that  for any $n\in\N$ 
 and there is a Hilbert  frame  $(x_i)_{i=1}^N$ of $\ell_2^n$, with $N\le qn$, so that
  $(x_i)_{i=1}^N$  satifies $(\vp,\delta,C)$-BCNQP.
\end{remark}

As in the previous section we formulate a corollary of Theorem \ref{T:5.6} for the infinite dimensional
situation.

\begin{cor}\label{C:5.8}
 Assume that  $X$ has the $\pi_\lambda$-property and that $\Pb=(P_n)$ is a sequence of uniformly bounded projections
 approximating point-wise the identity  on  $X$. Let  $(x_i,f_i)$
 be a frame of $X$, with $(x_i)$ being bounded, and assume  for  an increasing sequence $(L_n)\subset \N$,  $0<r,\delta,\vp <1$ and $C<\infty$ that
 the following conditions are satisfied:
 \begin{enumerate}
 \item[a)] For $n\in\N$ and $x\in X_n=P_n(X)$ 
 $$\Big\| x-\sum_{i=1}^{L_n} \la f_i,x \ra P_n(x)\Big\|\le (1-r)\|x\|,
 $$ 
 \item[b)]   $\|P_n(x_i)\|\ge r$, for all  $n\in\N$ and $i\le L_n$, and
 \item[c)] For all $n\in\N$ and all  $x\in\{\sum_{i=1}^{L_n} a_i x_i :|a_i|\le 1 \text{ for }i=1,2,\ldots L_n\} $ there is a sequence $(k_i)_{i\le L_n} $ so that $$\Big\|x-\sum_{i=1}^{L_n} \delta k_i x_i\Big\|<\vp \text{ and }
 \max_{i\le L_n} |k_i|\delta \le C$$
 (in particular $(x_i,f_i)$ satisfies the BCNQP).
 \end{enumerate}
Then there is either a constant $c>0$ so that $L_n \ge c\dim P_n(X) \ln(\dim P_n(X))$
 or  the spaces $\{\ell_\infty^n:n\in|N\}$ are uniformly contained in $X$.
\end{cor}
\begin{proof}
From   assumption (a)  it follows that the operators $(Q_n)$, with
$$Q_n: X_n\to X_n,\qquad x\mapsto \sum_{i=1}^{L_n}\la P_n^*(f_i),x\ra P_n(x_i ),$$
are uniformly bounded ($\|Q_n\|\le 2$, for $n\in\N$), invertible and their inverses are uniformly bounded
($\|Q_n^{-1}\|\le \frac1{r}$, for $n\in\N$). For $x\in X_n$ we write
\begin{align*}
x= Q_n^{-1}Q_n(x)
=\sum_{i=1}^{L_n}\la f_i,x\ra (Q_n^{-1}\circ P_n)(x_i ) ,\end{align*}
and deduce therefore that 
$$\big(y_i^{(n)},g_i^{(n)}\big)_{i=1}^{L_n}:=\big({(Q_n^{-1}\circ P_n)(x_i )}, f_i|_{X_n}\big)_{i=1}^{L_n}$$
 is a frame of
$X_n$. We now verify that for $n\in\N$  $\big(y_i^{(n)},g_i^{(n)}\big)_{i=1}^{L_n}$ satisfies the 
 $(\tilde \vp,\delta,C)$-BCNQP for some $\tilde\vp>0 $, which is independent of $n$. Indeed by assumption (c)   one can choose for $n\in\N$ and  
 $(a_i)_{i=1}^{L_n}\in[-1,1]$ 
 some  $(k_i)_{i=1}^{L_n}\subset \Z$ so that
$$\Big\|\sum_{i=1}^{L_n}  a_i x_i-\sum_{i=1}^{L_n} \delta k_i x_i\Big\|\le \vp,
 \text{ and } \max_{i\le L_n} |k_i|\le \frac{C}{\delta}
$$
and, thus,
\begin{align*}
\Big\|\sum_{i=1}^{L_n} a_i y_i^{(n)}-\sum_{i=1}^{L_n} \delta k_iy_i^{(n)}\Big\|&=
 \Big\|\sum_{i=1}^{L_n}  a_i\ {(Q_n^{-1}\circ P_n)(x_i )}-
 \sum_{i=1}^{L_n} \delta k_i\ {(Q_n^{-1}\circ P_n)(x_i )}\Big\|\\
&\le \vp\max_{i\le L_n}\|Q_n^{-1}\circ P_n\|\le \frac\vp{r}\sup_n\|P_n\|=:\tilde \vp.
\end{align*}
Then for $n\in\N$ and $i\le L_n$ it follows from assumption (b) that
$$ \frac{r}2\le \frac{r}{\|Q_n\|}\le \|y^{(n)}_i\|\le \|P_n\| \cdot\|Q^{-1}_n\|\cdot \|x_i\|\le 
\frac{ \sup_j\|P_j\|  \sup_j\|x_j\|}r<\infty.$$
Thus Theorem \ref{T:5.6}, Remark \ref{R:5.6a} and  Theorem \ref{T:5.4} yield our claim.
\end{proof}

By Remark \ref{R:5.2}, for  semi-normalized bases  $(x_i)$ (together with their
coordinate functionals) the properties BCNQP and NQP are equivalent. 
 We therefore deduce from Theorem \ref{T:5.6} the following
\begin{cor}\label{C:5.9}
An infinite dimensional  Banach space $X$ with non trivial cotype cannot have a semi-normalized basis having the NQP.

In particular (see Problem 5.18 in \cite{DOSZ}) $\ell_1$ does not have   a semi-normalized basis with the NQP.
\end{cor}

\begin{proof} Suppose $(x_i)$ is a semi-normalized basis  with the $(\vp,\delta)$-CQP. Then we can we can $(P_n)$ take to be the basis projections and $L_n=n$. By Corollary \ref{C:5.8}  does not have finite cotype.
\end{proof}

\section{Concluding remarks and open problems}\label{S:6}

Kashin's  \cite{Ka} celebrated result states that  for any 
 $\lambda>1$ there is a $K=K_\lambda$ so that 
 for any $n\in\N$ and any $N\ge \lambda n$, $N\in\N$, there is 
 an orthogonal projection $U$ from  $\R^N$ onto $\R^n$
 (i.e. $U$ is an $N$ by $n$ matrix whose rows are orthonormal) 
 so that 
 \begin{equation}\label{E:6.1}
 B_n\subset \frac{K}{\sqrt N} U( Q_N)\subset K B_n
 \end{equation}
(as before $B_n$ is the euclidean unit ball in $\R^n$ while 
 $Q_N$ is the unit cube in $\R^N$).

   Lyubarskii and Vershinin observed  in \cite{LV} that 
  the column vectors $(u_i)_{i=1}^N$  form a tight frame
  (with $A=B=1$), that  the first inclusion in \eqref{E:6.1}
  yields that every $x\in B_n$ can be written as
  $$x=\sum_{i=1}^N \frac{K}{\sqrt N} a_i u_i \text{ with } \|(a_i)\|_{\ell_\infty}\le 1,$$
  and that the second inclusion implies that the operator
  $\frac{1}{\sqrt N} U: \ell_\infty^N\to \ell_2^n$ is of norm not greater than $1$, and that
  therefore  for given $\vp>0$  there is a sequence $(k_i)_{i=1}^N\subset \Z\cup[-K/\vp,K/\vp]$,
  so that  $\max_{i\le N} |a_iK-k_i\vp|\le \vp$ and, thus,
  \begin{equation}\label{E:6.2}
  \Big\| x- \sum_{i=1}^N \vp k_i\frac{u_i}{\sqrt N}\Big\|\le \vp \text{ and }
  \max |\vp k_i |\le K. 
  \end{equation}
 
 Thus,  Kashin's orthogonal projections (which are actually chosen randomly),
   lead to a frame  $(x_i^{(n)})_{i\le N}=(u_i^{(n)}/\sqrt{N})_{i\le N}$ for $\ell_2^n$, whose length is not larger
   than a fixed multiple of $n$, and, for  which we can represent any element $x$ in $B_n$ 
    as a quantized  linear combination with bounded coefficients.
  Since the zonotope $\{\sum_{i=1}^N a_i x_i |a_i|\le 1 \text{ for }i=1,2\ldots N\}$ lies in $B_n$,
it follows that the  Hilbert frame  $(x_i^{(n)})_{i\le N}$ satisfies for any $\vp>0$ the
$(\vp,\vp, K)$-BCNQP.

     In view of the results presented in  sections \ref{S:4} and \ref{S:5} this 
      is  the best one could do in the finite dimensional case. We are therefore interested
       in extensions of this result by Lyubarskii and Vershinin to other spaces as well as the infinite dimensional space,
       \begin{prob}\label{Prob:6.1}    Does the above cited result hold for other finite dimensional spaces?  More precisely, assume that $0<\delta,\vp<1$, $C\ge 1$ are fixed. 
For which $n\in\N$ and which $n$-dimensional spaces $X$ can we find a frame
       $(x_i,f_i)_{i=1}^N$, with, say $N=2n$, so that for  any $x\in B_X$ there is a $(k_i)_{i=1}^N\subset \Z$ 
       so that
       $$   \Big\| x- \sum_{i=1}^N \delta k_i x_i\Big\|\le \vp \text{ and }
  \max |\delta k_i |\le C.$$
\end{prob}
\begin{remark} The above presented argument from \cite{LV} shows that if there is a quotient $Q:\ell^N_\infty\to X$,
 and a frame $(x_i,f_i)_{i=1}^N$ of $X$,  so that $Q(e_i)=x_i$, for $i=1,\ldots N$, 
 and so that for some $K<\infty$  $B_X\subset Q(B_{\ell_\infty})\subset K B_X$, then
 there is for  all $x\in B_X$ and all $\delta>0$ a sequence $(k_i)_{i\le N}\subset \Z$, so that
$$    \Big\| x- \sum_{i=1}^N \delta k_i x_i\Big\|\le \|Q\|\delta/2\le K\delta/2 \text{ and }
\max_{i\le N}|k_i|\le \frac1\delta.
$$
Conversely, assume that for some   $0<\delta,\vp<1$, $C\ge 1$ we can find
 for all $x\in B_X$ a sequence $(k_i)\subset\Z$  so that 
$$    \Big\| x- \sum_{i=1}^N \delta k_i x_i\Big\|\le\vp \text{ and }
\max_{i\le N}|k_i|\le \frac C\delta.
$$
Then we can choose  by induction for $x\in B_X$ a $z_i$ with
$$z_n=\sum_{i=1}^N k^{(n)}_i\delta x_i, (k_i)\subset \Z\cap [-C/\delta,C/\delta],$$
so that 
$$\Big\|x -\sum_{i=1}^n \vp^{i-1} z_i\Big\|\le \vp^n.$$
Indeed assuming $z_1,\ldots z_{n-1}$ have been chosen
we apply our assumption to $y=\vp^{1-n}\big[x -\sum_{i=1}^{n-1} \vp^{1-i} z_i\big]\in B_X$
to find $z_n$.

Thus, it follows that 
$$x=\sum_{i=1}^\infty \vp^{i-1} z_i=\sum_{j=1}^N x_j\sum_{i=1}^\infty \vp^{i_1} \delta k^{(i)}_j,$$
which means that there is a $C_1<\infty$ only depending on $\vp,\delta$ and $C$ so that
$$B_X\subset \Big\{\sum_{j=1}^N a_i x_i: |a_i|\le C_1 \Big\}.$$
If we define now 
$$Q:\ell_\infty^N\to B_X,\quad z\mapsto \sum_{i=1}^N C_1z_ix_i,$$
we deduce that $Q$ is a quotient map and that $B_X\subset Q(B_{\ell_\infty})$ but we cannot deduce
(at least not obviously) a bound for $\|Q\|$.
\end{remark}
 \begin{prob}
 Is there an infinite dimensional version of the result of Lyubarskii and Vershinin?
  I.e. for which  infinite dimensional Banach spaces $X$ with a basis $(e_i)$
   does there exist $0<\delta,\vp<1$, $C\ge 1$ and a frame 
   $(x_i,f_i)_{i\in\N}$ so that for any $x\in B_X$, $n=\max\supp(x)<\infty$,
     there is a $(k_i)_{i=1}^N$, with, say, $N\le 2n$,  so that
        $$   \Big\| x- \sum_{i=1}^N \delta k_i x_i \Big\|\le \vp \text{ and }
  \max |\delta k_i |\le C?$$
       \end{prob}

\end{document}